\documentclass[10pt]{amsart}
\RequirePackage[OT1]{fontenc}
\RequirePackage{amsthm,amsmath}
\RequirePackage[colorlinks,citecolor=blue,urlcolor=blue]{hyperref}
\usepackage{hhline}
\usepackage[table,xcdraw]{xcolor}
\usepackage{multirow}
\usepackage[toc,page,titletoc]{appendix}
\usepackage{xr-hyper}
\usepackage{hyperref}
\usepackage[british]{babel}
\usepackage{nameref}
\usepackage{amsfonts}
\usepackage{amssymb}
\usepackage{amsfonts}
\usepackage{mathrsfs}
\usepackage{amsfonts}
\usepackage{amstext}
\usepackage{amsopn}
\usepackage{pgfplots,geometry}
\usepackage{color}
\usepackage[]{subfig}
\usepackage[justification=justified]{caption}
\usepackage{tikz}
\usepackage{mhchem,gensymb}
\usepackage{graphicx}
\usepackage{chemfig}
\usepackage{url}
\usepackage{lmodern}
\usepackage{pifont}
\usepackage[abs]{overpic}
\usepackage{texdraw}
\usepackage{extarrows}
\usepackage{bbm}
\usepackage[utf8]{inputenc}
\usepackage{mathtools}
\usepackage{enumitem}

\allowdisplaybreaks
\usetikzlibrary{arrows, decorations.markings}
\usetikzlibrary{positioning}
\usetikzlibrary{fit}
\usetikzlibrary{backgrounds}
\usetikzlibrary{calc}
\usetikzlibrary{shapes}
\usetikzlibrary{mindmap}
\usetikzlibrary{decorations.text}
\pgfplotsset{compat=newest}

\numberwithin{equation}{section}
\theoremstyle{plain}
\newtheorem{theorem}{Theorem}[section]
\newtheorem{lemma}[theorem]{Lemma}
\newtheorem{corollary}[theorem]{Corollary}
\newtheorem{proposition}[theorem]{Proposition}

\newtheorem{question}[theorem]{Question}

\theoremstyle{definition}

\newtheorem{example}[theorem]{Example}
\theoremstyle{remark}
\newtheorem{remark}[theorem]{Remark}
\newcommand{\eb}{\begin{example}}
\newcommand{\ee}{\end{example}}
\newcommand{\eqb}{\begin{equation}}
\newcommand{\eqe}{\end{equation}}
\newcommand{\spb}{\begin{split}}
\newcommand{\spe}{\end{split}}
\newcommand{\cab}{\begin{cases}}
\newcommand{\cae}{\end{cases}}
\newcommand{\thmb}{\begin{theorem}}
\newcommand{\thme}{\end{theorem}}
\newcommand{\qb}{\begin{question}}
\newcommand{\qe}{\end{question}}
\newcommand{\prob}{\begin{proposition}}
\newcommand{\proe}{\end{proposition}}
\newcommand{\leb}{\begin{lemma}}
\newcommand{\lee}{\end{lemma}}
\newcommand{\cob}{\begin{corollary}}
\newcommand{\coe}{\end{corollary}}
\newcommand{\enb}{\begin{enumerate}}
\newcommand{\ene}{\end{enumerate}}
\newcommand{\cenb}{\begin{center}}
\newcommand{\cene}{\end{center}}

\newcommand{\prb}{\begin{proof}}
\newcommand{\pre}{\end{proof}}
\newcommand{\rb}{\begin{remark}}
\newcommand{\re}{\end{remark}}

\newcommand{\tS}{\text{S}}
\newcommand{\ka}{\kappa}

\newcommand{\lt}{\left}
\newcommand{\rt}{\right}

\newcommand{\om}{\omega}
\newcommand{\E}{\mathbb{E}}
\newcommand{\N}{\mathbb{N}}
\newcommand{\bP}{\mathbb{P}}

\newcommand{\R}{\mathbb{R}}
\newcommand{\Z}{\mathbb{Z}}

\newcommand{\cR}{\mathcal{R}}
\newcommand{\cS}{\mathcal{S}}
\newcommand{\cT}{\Omega}

\newcommand{\cX}{\mathcal{Y}}
\newcommand{\cY}{\mathcal{X}}

\newcommand{\rO}{\mathrm{O}}

\newcommand{\wY}{\widetilde{Y}}
\newcommand{\supp}{{\rm supp\,}}

\newcommand{\srcsize}{\@setfontsize{\srcsize}{5pt}{5pt}}

\makeatletter
\DeclareRobustCommand\widecheck[1]{{\mathpalette\@widecheck{#1}}}
\def\@widecheck#1#2{%
    \setbox\z@\hbox{\m@th$#1#2$}%
    \setbox\tw@\hbox{\m@th$#1%
       \widehat{%
          \vrule\@width\z@\@height\ht\z@
          \vrule\@height\z@\@width\wd\z@}$}%
    \dp\tw@-\ht\z@
    \@tempdima\ht\z@ \advance\@tempdima2\ht\tw@ \divide\@tempdima\thr@@
    \setbox\tw@\hbox{%
       \raise\@tempdima\hbox{\scalebox{1}[-1]{\lower\@tempdima\box
\tw@}}}%
    {\ooalign{\box\tw@ \cr \box\z@}}}
\makeatother
\allowdisplaybreaks
\begin{document}
\title[Classification of dynamics for 1-d CTMCs]
  {Full classification of dynamics for one-dimensional continuous time Markov chains with polynomial transition rates}
\author{Chuang Xu}
\address{
Faculty of Mathematics,
Technical University of Munich, 85748 Garching bei M\"{u}nchen, Germany.}
\email{Chuang.Xu@ma.tum.de {\rm (corresponding author)}}
\author{Mads Christian Hansen}
\author{Carsten Wiuf}
\address{
Department of Mathematical Sciences,
University of Copenhagen, Copenhagen,
2100, Denmark.}

\subjclass[2010]{}

\date{\today}

\noindent

\begin{abstract}
This paper provides full classification of the dynamics for continuous time Markov chains (CTMCs) on the non-negative integers with polynomial transition rate functions and without   arbitrary large backward  jumps. Such stochastic processes are abundant in applications, in particular in biology. More precisely, for CTMCs of bounded jumps, we provide necessary and sufficient conditions in terms of calculable parameters for explosivity, recurrence vs transience, positive recurrence vs null recurrence, certain absorption, and implosivity. Simple sufficient conditions for exponential ergodicity of stationary distributions  and quasi-stationary distributions as well as existence and non-existence of moments of hitting times are also obtained. Similar simple sufficient conditions for the aforementioned dynamics together with their opposite dynamics are established for CTMCs with unbounded {forward} jumps. The results generalize  criteria for birth-death processes by Karlin and McGregor in the 1960s.
Finally, we apply our results to stochastic reaction networks, an extended class of branching processes, a general bursty single-cell stochastic gene expression model, and population processes, none of which are birth-death processes. The approach is based on a mixture of Lyapunov-Foster type results, the classical semimartingale approach, {as well as estimates of stationary measures.}
\end{abstract}


\maketitle

\section{Introduction}

Continuous time Markov chains (CTMCs) on a countable state space are widely used in applications, for example, in genetics \cite{E79}, epidemiology \cite{PCMV15}, ecology \cite{G83}, biochemistry and systems biology \cite{W06}, sociophysics \cite{WH83}, and queueing theory \cite{GH98}.
For a CTMC on a countable state space, criteria for dynamical properties (explosivity, recurrence, certain absorption, positive recurrence, etc.) are among the fundamental topics and areas of interest.

 {A primary source of inspiration for our work comes from  stochastic reaction network (SRN) theory, where examples are abundant. In the present context, SRNs are CTMC  models of    \emph{(chemical) reaction networks} with polynomial transition rates \cite{AK15} (see Section \ref{sec:RN} for a precise definition).  In particular, we are interested in one-species reaction networks, where the reactions take the form $n\tS\ce{->[\kappa]}m\tS$ for two non-negative integers, $n,m$, and $\ka>0$, a positive  reaction  rate constant.  Here $\tS$ represents a (chemical) species common to all reactions in the network, and the reaction represents the conversion of $n$ molecules of the species $\tS$ into $m$ molecules of the same species. Each reaction has a transition rate, a propensity to `fire'. The transition rate of $n\tS\ce{->[\kappa]}m\tS$ is   $\eta(x)=\kappa x(x-1)\ldots(x-n+1)$, $x\in\N_0$. Whenever  the reaction fires, the corresponding Markov chain on the state space $\N_0$ jumps from the current state $x$ to the state $x+m-n$, the number of $\tS$ molecules after the firing of the reaction. Different reactions may contribute to the same transition in the state space.}

{While the chemical terminology may suggest a restricted usage of such models, this is  by far not so.  In fact, SRNs have widespread use in the sciences by interpreting species as agents, individuals, and similar entities, and reactions as interactions among these \cite{AK15}. One might emphasize the SIR model in epidemiology as a particular example \cite{PCMV15}.}

{Consider the following two examples of one-species SRNs from the recent literature, consisting of seven and five reactions, respectively,
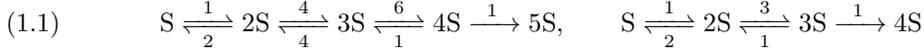
\eqb\label{Eq-1}
\text{S}\ce{<=>[1][2]} 2\text{S}\ce{<=>[4][4]} 3\text{S}\ce{<=>[6][1]} 4\text{S}\ce{->[1]} 5\text{S},\qquad \text{S}\ce{<=>[1][2]}  2\text{S}\ce{<=>[3][1]}  3\text{S}\ce{->[1]}  4\text{S}
\eqe
 \cite{ACKK18}. A  key issue is to understand whether  the graphical representation of the reaction networks   determine  the  dynamics of the corresponding CTMCs, irrespectively, their initial values.}
The first network is {\em explosive} {(except if the initial state is $0$, which forms a singleton communicating class)}, while the second is {\em positive recurrent} {on the positive integers (again $0$ forms a singleton class)} \cite{ACKK18}, which might be inferred from known birth-death process (BDP) criteria  \cite{A91}.
However, these criteria are \emph{not} computationally simple and blind to the {graphical} structure of the networks. A simple  explanation for the drastic difference in the dynamics of these two  random walks on $\N_0$ is desirable but remains unknown  \cite{ACKK18}.

Motivated by the above concern, we provide criteria for dynamical properties of \,CTMCs on $\N_0$ with polynomial-like transition rates {and without the possibility of arbitrary large negative jumps; as in the examples above}. These CTMCs are ubiquitous in applications,
and  approximate  almost all CTMCs due to the polynomials being dense in the space of continuous functions  \cite{B76}.
Specifically, we provide simple threshold criteria for the existence and non-existence of moments of hitting times, positive recurrence and null recurrence, and exponential ergodicity of stationary distributions and quasi-stationary distributions (QSDs)  in terms of  \emph{four} easily computable parameters, {derived from the transition rates}. Additionally, we provide  necessary and sufficient conditions for explosivity, recurrence (\emph{vs} transience), certain absorption, and implosivity. {These conditions provide simple explanations for the dynamical discrepancies of the two SRNs in \eqref{Eq-1}.}

{Our approach is to apply the classical semi-martingale approach used in Lamperti's problem \cite{L60}, as well as Lyapunov-Foster theory \cite{CV17,MP14,MT93} with delicately constructed Lyapunov functions (in particular, we make use of the techniques in \cite{MP14}). } 
The problem of finding neat and desirable necessary and sufficient conditions for  dynamical properties of  CTMCs, 
has existed for a long time \cite{MT93,MT09,CV17}. However, the fact that this has not been accomplished yet, indicates that {it might be a non-trivial task.}  A main contribution of this paper is to \emph{identify} a large class of CTMCs (without a built-in detailed balanced structure) for which computationally simple, sufficient and necessary criteria can be established for dynamical properties of interest. Our criteria save the effort of constructing Lyapunov functions and applying Lyapunov-Foster theory case by case. Also, a case by case approach is ignorant of the underlying {graphical} structure of the Markov chain.

The simple necessary and sufficient conditions for the dynamical properties are determined by calculating up to four parameters, $R$, $\alpha$, $\beta$ and $\gamma$, that are  expressed in terms of the coefficients of the first two terms of the polynomial-like transition rate functions (the specific assumptions are given in Section \ref{sect-preliminary}). For illustration, let $\cT$ be the set of jump sizes,   and
\begin{equation}\label{asymptotic-transition-rate}
\lambda_{\om}(x)=a_{\om} x^{d_{\om}}+b_{\om}x^{d_{\om}-1}+\rO(x^{d_{\om}-2}),\quad \om\in\cT,
\end{equation}
be the transition rate functions,  where $d_\omega$ is the degree of  $\lambda_{\om}$ {and $\rO$ is Landau's symbol.} 
Define  $R=\max_{\om\in\cT}d_{\om}$ and
\begin{equation}\label{formula-alpha-gamma}
  \alpha=\!\!\!\sum_{\om\colon d_{\om}=R}a_{\om}\om,\quad \gamma=\!\!\!\sum_{\om\colon d_{\om}=R}b_{\om}\om+\!\!\!\sum_{\om\colon d_{\om}=R-1}a_{\om}\om, \quad
\beta=\gamma-\frac{1}{2}\sum_{\om\colon d_{\om}=R}a_{\om}\om^2.
\end{equation}
Based on these {four} parameters, a \emph{full} classification of the dynamical properties  can be achieved (see Theorems~\ref{th-7}, \ref{th-8}, \ref{th-9}, and \ref{th-10}) and is summarized in Table~\ref{table-summary} below. {The   parameters $\alpha,\beta,\gamma$ only depend on the coefficients of the monomials of degree $R$ and $R-1$. Furthermore, the parameter $\alpha$  might be interpreted as a   sum  over the jump sizes, weighted by the coefficients of the monomials of degree $R$. Similarly,  $\gamma$ might  be interpreted as a  sum over the jump sizes, weighted by the coefficients of the monomials of degree $R-1$.}

\captionsetup[table]{}
\arrayrulecolor{black}
\begin{table}[ht]
\centering
\setlength{\tabcolsep}{9pt}
\renewcommand{\arraystretch}{1.3}
\begin{tabular}{  |p{1.05cm}|p{.75cm}|p{.95cm}|p{.95cm}|p{1.8cm}|p{.95cm}|p{.95cm}|p{.75cm}| }
\cline{2-8}
\multicolumn{1}{c|}{}&\multirow{2}{*}{$\alpha<0$}&\multicolumn{5}{c|}{$\alpha=0$}&\multirow{2}{*}{$\alpha>0$}\\\hhline{~~-|-|-|-|-|~}
\multicolumn{1}{c|}{}&\multicolumn{1}{c|}{}&$\gamma<0$&$\gamma=0$&$\beta<0<\gamma$&$\beta=0$&$\beta>0$& \multicolumn{1}
{c|}{}\\\hline
$R=0$ &\multicolumn{1}{c|}{\cellcolor{red}} &\multicolumn{1}{c|}{\cellcolor{black}}&\multicolumn{1}{c|}{\cellcolor{blue}}&\multicolumn{2}{c|}{\cellcolor{black}}&\multicolumn{2}{c|}{\cellcolor{green}} \\\hhline{|-|-|-|~|-|-|~~|}
$R=1$ &\multicolumn{1}{c|}{\cellcolor{red}ES}&\cellcolor{red}&\multicolumn{3}{c|}{\cellcolor{blue}NS/NQ}&\multicolumn{2}{c|}{\cellcolor{green}NS/NQ} \\\hhline{|-|-|~|-|-|~|~|-|}
$R=2$ &\cellcolor{pink}&\multicolumn{3}{c|}{\cellcolor{red}}&\multicolumn{1}{c|}{\cellcolor{blue}}&
\multicolumn{1}{c|}{\cellcolor{green}} &\multicolumn{1}{c|}{\cellcolor{yellow}} \\\hhline{|-|~|-|-|-|-|-|~|}
$R>2$ &\multicolumn{5}{c|}{\cellcolor{pink}ES/UQ}&\multicolumn{2}{c|}{\cellcolor{yellow}NS/NQ} \\\hline
\end{tabular}
\vspace{3.5pt}
\caption{\small Parameter regions with different dynamical properties.
Implosive (pink), positive recurrent but non-implosive (red), null recurrent (blue),  transient and non-explosive (green), explosive (yellow), and not possible parameter combinations (black). ES=exponential ergodicity of stationary distribution, UQ=uniform exponential ergodicity of QSD, NQ=no QSDs, NS=no ergodic stationary distributions. The {parameter} regions below $\alpha=0$ assumes $\Omega$ is finite.\label{table-summary}}
\end{table}

To see the power of our results, consider the following SRN, which is not a BDP,
\begin{equation}\label{SRN-2}
0\ce{<=>[\ka_1][\ka_{2}]}m\tS\ce{<=>[\ka_3][\ka_{4}]}(m+1)\tS\ce{->[\ka_5]}(m+3.0)\tS,
\end{equation}
where {$m$ is a positive integer}, and $\kappa_i$, $i=1,\ldots,$, is a positive rate constant. 
Then,  $R=m+1$, $\alpha=2\ka_5-\ka_{4}$, $\beta={\ka_3-m\ka_{2}+\tfrac{m^2+m-1}{2}\kappa_4-(m^2+m+2)\kappa_5}$,  and  $\gamma={\kappa_3-m\kappa_2+\tfrac{m(m+1)}{2}\kappa_4}$ ${-m(m+1)\kappa_5}$. The criteria  established in Section~\ref{sect-criteria} (and collected in Table \ref{table-summary}) implies that the SRN is (in the sense of the underlying irreducible CTMC on $\N_0$) 
\begin{itemize}
\item[(a)] explosive a.s. if and only if (i) $\alpha>0$ or (ii) $\alpha=0$, $\beta>0$, $R>2$, and non-explosive if either (i) or (ii) fails.
\item[(b)] recurrent if and only if (iii) $\alpha<0$ or (iv) $\alpha=0$, $\beta\le0$, and transient if and only if both (iii) and (iv) fail.
\item[(c)] positive recurrent  if and only if (iii), (v) $\alpha=0$, $\beta<0$, or (vi) $\alpha=0$, $\beta=0$, $R>2$ holds, and null recurrent if and only if (vii) $\alpha=0$, $\beta=0$, $R=2$.
\item[(d)] implosive\footnote{Here implosive means positive recurrent with uniformly bounded expected first return time. See Subsection~\ref{subsec-implosive} for the precise definition.} if and only if (iii) or (viii) $\alpha=0$, $\beta\le0$, $R>2$ holds, while it is non-implosive if and only if both (ii) and (viii) fail.
\end{itemize}
The above example shows the applicability and simplicity of our results, {in fact, the computations could easily be implemented in a software program that takes a reaction network as input and outputs the network dynamical properties. Furthermore, the example illustrates the richness of the dynamical properties that might reside within a single example by varying the parameters of the model.} All possibilities for $\alpha, \beta$ and $R$ are covered {(the parameter $\gamma$ is irrelevant for SRNs \cite{WX20}).}
The stability of the chain  only depends on $\alpha$ (which is independent of $m$), unless  $\alpha=0$, in which case the sign of $\beta$ determines the stability. If so, then $\beta=\ka_3-m\ka_2-3\ka_5$, {which depends on $m$. Thus,   if $\ka_3-\ka_2-3\ka_5>0$, then by choosing $m>1$ large enough the stability of the chain flips.  The parameter  $\alpha$ plays a role similar  to that of the largest Lyapunov exponent for $\alpha\not=0$. Analogously, when $\alpha=0$, the parameter $\beta$ determines the stochastic stability  and hence, plays a role similar  to that of  the second largest Lyapunov exponent in the critical case. }

\subsection*{Brief description of our approach}

Although the approach essentially is based on Lyapunov-Foster type results, the sharp criteria for diverse dynamical properties of CTMCs are established by combining a mixture of results \cite{MT93,C97,AI99,MP14,CV17}, {in particular \cite{MP14} provides useful criteria}. 

The most prominent difficulty in deriving necessary and sufficient conditions for dynamical properties of general CTMCs, with  multiple jump sizes, lie in the non-calculability of stationary distributions/measures, as well as the non-existence of  orthogonal polynomials  \cite{KM57}. This also explains why in general, a partial result in terms of a sufficient, but not a necessary condition, by construction of a Lyapunov  function, is likely. Here we  discover that Lyapunov-Foster theory {and the semimartingale approach indeed are} enough to derive necessary and sufficient conditions.
To obtain   conditions that are not only sufficient but also necessary, we check if the negation of a condition is also sufficient for the reverse dynamical property. {Moreover, to show null recurrence, we also rest on estimates for stationary measures in \cite{AI99}.} Finally, we like to point out that some of the Lyapunov functions we use, appear to be rarely used  in the literature.

\subsection*{Comparison with results in the literature}
Complete classification of dynamical properties seems quite rare in the literature. Here, we summarize relevant results together with the methods applied.

Reuter provided necessary and sufficient conditions for explosivity  of CTMCs  (the so-called Reuter's criterion) \cite{R57}, but these conditions are difficult to check, except in special cases, e.g. BDPs \cite{KM57}, and competition processes \cite{R61}.
This is due to the fact that the conditions involve infinitely many algebraic equations.

Karlin and McGregor  established  threshold results for explosivity, recurrence, as well as certain absorption of BDPs with and without absorbing states, by means of   the so-called Karlin-McGregor integral representation formula \cite{KM57}. The  existence of such  formula is essentially due to the tridiagonal structure of the $Q$-matrix. For the same reason, it is delicate to extend such an approach to \emph{generalized BDPs}: Pure birth processes  \cite{C99}, one-sided skip free CTMCs \cite{C97,C04,CPZC05}, and  recently (higher-dimensional) quasi-birth-death processes (QBDPs) with tridiagonal block structure of the $Q$-matrix \cite{FD21}.

In the context of QSDs, there are few  threshold results for certain absorption, existence and uniqueness as well as quasi-ergocidity of QSDs.  van Doorn \cite{V85,V91} obtained ergodicity,   existence and non-existence as well as uniqueness of QSDs for absorbed BDPs, also building on the Karlin-McGregor integral representation formula.
Later, Ferrari et al. \cite{FKMP95} generalized the results in \cite{V91}. They derived a necessary and sufficient condition for the existence of a QSD  on the positive integers for which zero is an absorbing state, using the so-called renewal dynamical approach, assuming the CTMC is non-explosive, and that the absorption time is finite and unbounded with probability one. Then, the existence of a QSD is equivalent to  finiteness of the exponential moment of the absorption time, for one {(and hence all)} initial transient states. But such a moment condition is again not straightforward to verify either, pending the assumptions.

{To sum up,  general  checkable  threshold criteria   for dynamical properties of CTMCs (absorbed and non-absorbed), other than generalized BDPs, are few. We identify a class of CTMCs with polynomial-like transition rates and without  arbitrary large backward jumps for which simple, checkable criteria for absorbed and non-absorbed CTMCs are found, based on the   coefficients  of the  polynomials.   The price for this is to impose some further mild regularity conditions, in addition to the two requirements mentioned above. }

\subsection*{Impact of our work and further extensions}

From the theoretical perspective,
\begin{enumerate}
\item[$\bullet$] the sufficient condition for the existence of  ergodic stationary distributions and QSDs allows us to further investigate the tail asymptotics of these distributions \cite{XHW20c} and the computation of these distributions (in a forthcoming paper).
\item[$\bullet$] the novel combination of the approaches presented here can further be extended to establish criteria for the dynamics of 1-d CTMCs with asymptotic polynomial transition rates, and higher dimensional CTMCs with $Q$-matrix of a certain block structure, in analogy with QBDPs and BDPs.
\item[$\bullet$] {a deeper understanding  of the threshold parameters
may provide insight into the dynamics of higher dimensional CTMCs on lattices.}
\end{enumerate}
From the perspective of applications,
\begin{enumerate}
  \item[$\bullet$] the criteria can be applied to completely classify the dynamics of 1-d mass-action SRNs, and in particular, we can prove the so-called Positive Recurrence Conjecture   \cite{AK18} for weakly reversible reaction networks in 1-d {\cite{WX20}.} 
  \item[$\bullet$] the criteria can be used to establish bifurcations of 1-d SRNs (in a forthcoming paper).
\end{enumerate}

\subsection*{Outline} In Section~\ref{sect-preliminary}, the notation and standing assumptions are introduced.  Section~\ref{sect-criteria} develops  threshold criteria for dynamical properties  of  CTMCs. Applications to  SRNs, a class of branching processes, a general bursty single-cell stochastic gene expression model, and population processes of non-birth-death process type, are provided in Section~\ref{sec5}.  Proofs of the main results are provided in Section \ref{sec:proofs}. Additional tools used in the proofs as well as proofs of some elementary propositions are appended.

\section{Preliminaries and assumptions}\label{sect-preliminary}

Let $\R$, $\R_{\ge 0}$, $\R_{>0}$ be the set of real, non-negative real, and positive real  numbers, respectively. Let $\Z$ be the set of integers, $\N=\Z\cap \R_{>0}$ and $\N_0=\N\cup\{0\}$. For $x,y\in \N$, let $x^{\underline{y}}=x(x-1)\cdots(x-y+1)$ be the descending factorial of $x$.

Let $(Y_t\colon t\ge 0)$ (or $Y_t$ for short) be a CTMC on a  closed,  { infinite}  state space $\cX\subseteq\N_0$ with conservative transition rate matrix $Q=(q_{x,y})_{x,y\in\cX}$, that is, every row sums to zero. A set $A\subseteq\cX$ is \emph{closed} if $q_{x,y}=0$ for all $x\in A$ and $y\in\cX\setminus A$ \cite{N98}. Assume the \emph{absorbing set} $\partial\subsetneq\cX$ is finite (potentially empty) and closed.  {Hence, $\cX\setminus\partial$ is unbounded.}

Let $\cT=\{y-x\colon q_{x,y}>0,\ \text{for some}\ x,y\in\cX\}$ {be the set of jump sizes}.
For  $\om\in\cT$, define the transition rate function:
$$\lambda_{\om}(x)=q_{x, x+\om},\quad x\in\cX.$$
Let $\cT_{\pm}=\{\omega\in\cT\colon {\rm sgn}(\om)=\pm1\}$ be the sets of forward and backward jump sizes, respectively.
Throughout, we assume the following {regularity conditions:}

\medskip
\noindent($\rm\mathbf{A1}$) $\cT_+\neq\varnothing$, $\cT_-\neq\varnothing$.
\medskip

\noindent($\rm\mathbf{A2}$) $\#\cT_-<\infty$.

\medskip
\noindent($\rm\mathbf{A3}$) $\sum_{\om\in\cT}\lambda_{\om}(x)|\om|<\infty,$ for all $x\in\cX$.
\medskip

\noindent($\rm\mathbf{A4}$) There exist $u, M\in\N$ such that $\lambda_{\omega}$ is a strictly positive polynomial of degree $\le M$ on the 
set $\cX\setminus\{0,\ldots,u-1\}$, for all  $\omega\in\cT$.

\medskip
\noindent($\rm\mathbf{A5}$)  $\cX\setminus\partial$ is irreducible.

\smallskip

If either $\cT_+=\varnothing$ or $\cT_-=\varnothing$, then $Y_t$ is a pure birth or death process (possibly with multiple jump sizes). 
Classification of states as well as the dynamics of such processes are  simpler than under ($\rm\mathbf{A1}$). Indeed, one can derive parallel results from the corresponding results under ($\rm\mathbf{A1}$).
Assumption ($\rm\mathbf{A2}$) implies $Y_t$ cannot make arbitrary large negative jumps.
Assumption ($\rm\mathbf{A3}$) is a  regularity condition  that  ensures functions like $x$, $\log x$ and $\log\log x$ are  in the domain of the infinitesimal generator of the CTMC, in order to serve as Lyapunov functions.
If $\cT$ is finite, then  ($\rm\mathbf{A2}$) and ($\rm\mathbf{A3}$) are automatically fulfilled.   In that case, the sums above are trivially polynomials for large $x$.

{Assumption ($\rm\mathbf{A4}$) implies
the Markov chain can make all  jumps in $\cT$ with positive probability from any `large' state $x\in\cX, x\ge u$.  Moreover,  ($\rm\mathbf{A3}$) and ($\rm\mathbf{A4}$)  together imply that $\sum_{\omega\in\cT}\lambda_{\omega}(x)$ and $\sum_{\omega\in\cT}\lambda_{\omega}(x)\om$ are polynomials of degree $\le M$ for $x\in\cX\setminus\{0,\ldots,u-1\}$ (Proposition~\ref{Spro-1}).  If ($\rm\mathbf{A4}$) fails, simple examples show that $\sum_{\omega\in\cT}\lambda_{\omega}(x)$ and $\sum_{\omega\in\cT}\lambda_{\omega}(x)\om$ may not be polynomials. That these sums are polynomials for large states is an essential property we rely on in proofs.}

{Assumption  ($\rm\mathbf{A4}$) is common in applications, especially in the context of chemical reaction networks and population processes \cite{AK11,EK09}.
Also assumption ($\rm\mathbf{A5}$) is standard  and generally satisfied in applications \cite{MP14,CV17}, potentially by restricting the state space. One can show that ($\rm\mathbf{A4}$) and ($\rm\mathbf{A5}$) together implies $\cX\setminus\partial$ is infinite, thus the assumptions are not compatible with a finite state space.}

{With the above assumptions,}  the following three parameters are well-defined and finite,
$$R=\max\{\deg(\lambda_{\omega})\colon \omega\in\cT\},\quad \alpha=\lim_{x\to\infty}\frac{\sum_{\omega\in\cT}\lambda_{\omega}(x)\omega}{x^R},\quad \gamma=\lim_{x\to\infty}\frac{\sum_{\omega\in\cT}\lambda_{\omega}(x)\omega-\alpha x^R}{x^{R-1}}.$$
If $R=0$, then trivially $\gamma=0$.
In particular,  if $\cT$ is finite, the {following additional parameter} is also well-defined and finite, \[\beta=\gamma-\frac{1}{2}\lim_{x\to\infty}\frac{\sum_{\omega\in\cT}\lambda_{\omega}(x)\omega^2}{x^R},\]
with $\beta<\gamma$. {The parameter $\alpha$ encodes the sign of the average jump size of the chain. }
It is straightforward to verify {that    \eqref{formula-alpha-gamma} is a consequence of the above parameter definitions, }
 due to the asymptotic expansions \eqref{asymptotic-transition-rate} of the transition rate functions.

\begin{example}
  Recall example \eqref{SRN-2} in the introduction,
  \begin{equation*}
0\ce{<=>[\ka_1][\ka_{2}]}m\tS\ce{<=>[\ka_3][\ka_{4}]}(m+1)\tS\ce{->[\ka_5]}(m+3)\tS,
\end{equation*}
Then $\cT=\{1,2,-1,m,-m\}$, $m\ge 1$, and 
\begin{align*}
\lambda_{1}(x) &= \left\{\begin{array}{rlc} \kappa_3 x^{\underline{m}}, & \quad \text{if}\ m> 1 \\
\kappa_3x^{\underline{m}}+\kappa_1, &  \quad \text{if}\ m=1   \end{array}\right\} {= \kappa_3x^m+\rO(x^{m-1}),}\\
\lambda_{2}(x) &=\left\{\begin{array}{rlc}\kappa_5 x^{\underline{m+1}}, &  \quad \text{if}\ m\neq 2  \\
\kappa_5x^{\underline{m+1}}+\kappa_1,&  \quad \text{if}\ m=2 
\end{array}\right\} {=\kappa_5x^{{m+1}}-\kappa_5\tfrac{m(m+1)}{2}x^{{m}}+\rO(x^{m-2}),} \\
\lambda_{-1}(x) &= \left\{\begin{array}{rlc} \kappa_4 x^{\underline{m+1}} &= \ \kappa_4x^{m+1}-\kappa_4\tfrac{m(m+1)}{2}x^m+\rO(x^{m-1}),&  \quad \text{if}\ m>1, \\
\kappa_4 x^{\underline{m+1}}+\kappa_2x^{\underline{m}} & =\ \kappa_4x^{m+1}+(\kappa_2-\kappa_4\tfrac{m(m+1)}{2})x^m+\rO(x^{m-1}),&  \quad \text{if}\ m=1,
\end{array}\right. \\
\lambda_m(x) &=\begin{array}{rc} \kappa_1, &\quad \text{if}\ m>2,\end{array} \\
\lambda_{-m}(x)& =\begin{array}{rlc} {\kappa_2}x^{\underline{m}} &=\ {\kappa_2} x^m+\rO(x^{m-1}),&\quad \text{if}\ m>1.\end{array}
\end{align*}
  Hence, $R=m+1$, and  
  $$\alpha=2\ka_5-\ka_{4},\quad \gamma={\ka_3-m\kappa_2+\tfrac{m(m+1)}{2}\kappa_4-m(m+1)\kappa_5},$$
  $$\beta={\ka_3-m\ka_{2}+\tfrac{m^2+m-1}{2}\kappa_4-(m^2+m+2)\kappa_5},$$
  by \eqref{formula-alpha-gamma}.  
 As mentioned in the introduction, the sign of $\alpha$ determines the stochastic stability of the CTMC \cite{MT93}.  
Hence,  $\alpha$ plays a similar role as the largest Lyapunov exponent. Analogously, when $\alpha=0$, the parameter $\beta$ determines the stochastic stability  and hence, plays a similar role to that of  the second largest Lyapunov exponent in the critical case.
\end{example}

\section{Criteria for dynamical properties}\label{sect-criteria}

In this section, we provide threshold criteria for various dynamical properties in terms of $R, \alpha, \beta, \gamma$. {Proofs are relegated to Section \ref{sec:proofs}.} {For ease of comparison, we collect all parameter conditions used in the main theorems below. These are listed in   the order they appear in the main theorems, see Table \ref{tab:conditions}. Figure \ref{fig:conditions} shows implications among the twenty one conditions.}

\bigskip
\begin{table}[]\begin{center}{\begin{tabular}{clcl}
 (C1) & $\alpha>0$,  $R>1$,  & (C2) &  $\alpha=0$, $\beta>0$, $R>2$, \\
(C3) & $\alpha<0$, & (C4) & $R\le1$, \\
(C5) &  $\alpha=0$, $R=2$, & (C6) & $\alpha=0,$ $\beta\le0$, \\
(C7) & $\alpha>0$, & (C8) & $\alpha=0,$ $\beta>0$, \\
(C9) &  $\alpha=0$, $\beta<0$, $R>1$, & (C10) &  $\alpha=\beta=0$, $R>2$, \\
(C11) &  $\alpha=0$, $\gamma<0$, $R=1$, & (C12) & $\alpha=0$, $\beta\le0$, $\gamma>0$, $R=1$,  \\
(C13) &  $\alpha=0$, $R=0$, & (C14) &  $\alpha=0$, $\beta<0$, $R=1$, \\
(C15) & $\alpha=0$, $\beta\le0$, $R=1$,  & (C16) & $\alpha=\beta=0$, $R=2$,  \\
(C17) & $\alpha<0$, $R\ge1$,  & (C18) & $\alpha=0$, $\beta\le0$, $R>2$, \\
(C19) &  $\alpha<0$, $R>1$, & (C20) & $\alpha<0$,  $R\le1$, \\
(C21) &  $\alpha=0$, $\beta<0$, $R=2$.
\end{tabular}}
\end{center}
\caption{{Labelling of the conditions in the main theorems in Section \ref{sect-criteria}.}}\label{tab:conditions}
\end{table}

\newcommand*{\NEarrow}{\rotatebox[origin=c]{30}{\(\Rightarrow\)}}
\newcommand*{\NWarrow}{\rotatebox[origin=c]{150}{\(\Rightarrow\)}}
\newcommand*{\SEarrow}{\rotatebox[origin=c]{330}{\(\Rightarrow\)}}
\newcommand*{\SWarrow}{\rotatebox[origin=c]{210}{\(\Rightarrow\)}}

\newcommand*{\SEEarrow}{\rotatebox[origin=c]{315}{\(\Rightarrow\)}}
\newcommand*{\SWWarrow}{\rotatebox[origin=c]{225}{\(\Rightarrow\)}}
\newcommand*{\SEEEarrow}{\rotatebox[origin=c]{335}{\(\Longrightarrow\)}}
\newcommand*{\SWWWarrow}{\rotatebox[origin=c]{205}{\(\Longrightarrow\)}}

\newcommand*{\LONGSEarrow}{\rotatebox[origin=c]{310}{\(\Longrightarrow\)}}

\color{black}
\begin{figure}
\begin{center}
\begin{tikzpicture}[scale=0.5]
\node  (a1) at (-2.5,0) {$\alpha>0$:};
\node  (A1) at (0,0) {(C1)};
\node  (A2) at (3,0) {(C7)};
\node  (I0) at (1.5,0) {$\Rightarrow$};

\node  (a3) at (-2.5,-4.5+3.0) {$\alpha=0$:};
\node  (D1) at (2.5,-9+3.0) {(C4)};
\node  (D2) at (2.5,-4.5+3.0) {(C11)};
\node  (D3) at (0,-6+3.0) {(C12)};
\node  (D4) at (2.5,-6+3.0) {(C14)};
\node  (D5) at (5,-6+3.0) {(C10)};

\node  (D6) at (0,-7.5+3.0) {(C13)};
\node  (D7) at (2.5,-7.5+3.0) {(C15)};
\node  (D8) at (5,-7.5+3.0) {(C18)};
\node  (D9) at (7.5,-7.5+3.0) {(C9)};
\node  (D10) at (10,-7.5+3.0) {(C16)};

\node  (D11) at (6.25,-9+3.0) {(C6)};
\node  (D12) at (12.5,-9+3.0) {(C5)};
\node  (D13) at (10,-6+3.0) {(C21)};
\node  (D14) at (15,-6+3.0) {(C2)};
\node  (D15) at (15,-7.5+3.0) {(C8)};

\node  (I1) at (2.5,-8.25+3.0) {$\Downarrow$};
\node  (I2) at (2.5,-5.25+3.0) {$\Downarrow$};
\node  (I3) at (1.25,-6.75+3.0) {$\SEarrow$};
\node  (I4) at (1.25,-8.25+3.0) {$\SEarrow$};
\node  (I5) at (2.5,-6.75+3.0) {$\Downarrow$};
\node  (I6) at (5,-6.75+3.0) {$\Downarrow$};
\node  (I7) at (3.9,-8.25+3.0) {$\SEEEarrow$};
\node  (I8) at (5.5,-8.25+3.0) {$\SEEarrow$};
\node  (I9) at (7,-8.25+3.0) {$\SWWarrow$};
\node  (I10) at (8.6,-8.25+3.0) {$\SWWWarrow$};
\node  (I11) at (10,-6.75+3.0) {$\Downarrow$};
\node  (I12) at (8.75,-6.75+3.0) {$\SWarrow$};
\node  (I13) at (11.24,-8.25+3.0) {$\SEarrow$};
\node  (I15) at (11.7,-7.5+3.0) {$\LONGSEarrow$};

\node  (I14) at (15,-6.75+3.0) {$\Downarrow$};

\node  (a2) at (-2.5+0.0,0-7.5) {$\alpha<0$:};
\node  (B1) at (0+0.0,0-7.5) {(C19)};
\node  (B2) at (3+0.0,0-7.5) {(C17)};
\node  (B3) at (6+0.0,0-7.5) {(C3)};
\node  (J1) at (1.5+0.0,0-7.5) {$\Rightarrow$};
\node  (J2) at (4.5+0.0,0-7.5) {$\Rightarrow$};
\node  (B4) at (3+0.0,-1.5-7.5) {(C20)};
\node  (B5) at (6+0.0,-1.5-7.5) {(C4)};
\node  (J3) at (4.5+0.0,-1.5-7.5) {$\Rightarrow$};
\node  (J4) at (4.5+0.0,-0.75-7.5) {$\NEarrow$};

\end{tikzpicture}
\end{center}
\caption{{Flow diagram of implications among the twenty one conditions.}}\label{fig:conditions}
\end{figure}
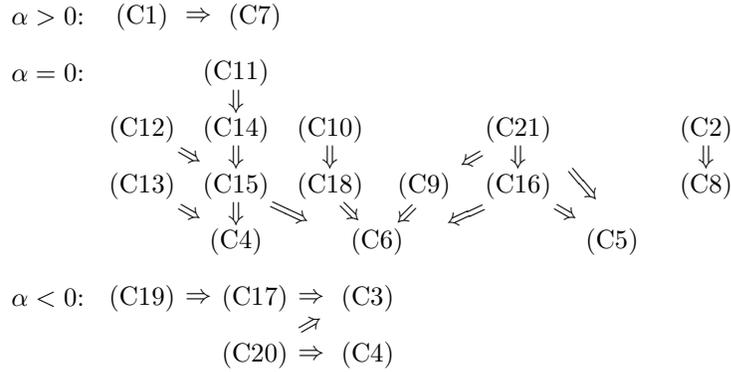

\color{black}

\subsection{Explosivity and non-explosivity}

The sequence $J=(J_n)_{n\in\N_0}$ of  {\em  jump times} {of a  CTMC $Y_t$} are defined by $J_0=0$, and $J_n=\inf\{t\ge J_{n-1}\colon Y_t\neq Y_{J_{n-1}}\}$, $n\ge1$, where $\inf\varnothing=\infty$ by convention.
The {\em life time} is denoted by $\zeta=\sup_{n}J_n$.
The process $Y_t$ is said to {\em explode} (with positive probability) at $y\in\cX$ if $\bP_y(\{\zeta<\infty\})>0$. In particular, $Y_t$ explodes almost surely (a.s.) at $y\in\cX$  if $\bP_y(\{\zeta<\infty\})=1$, and does   not  explode at $y\in\cX$ if $\bP_y(\{\zeta<\infty\})=0$ \cite{MT93}.
Hence, $Y_t$ does not explode if $Y_0\in\partial$ (since $\partial$ is closed and finite), and $\E_y(\zeta)<\infty$ implies  $Y_t$ explodes at $y$ a.s.
Recall that non-explosivity and explosivity are class properties. They hold for either all or no states in $\cX\setminus\partial$.
Hence, we simply say $Y_t$ is \emph{explosive} (explosive a.s., respectively) if it explodes  with positive probability (explodes a.s., respectively) at some state in $\cX\setminus\partial$,
and $Y_t$ is \emph{non-explosive} if it does not explode at some state in $\cX\setminus\partial$.

We present necessary and sufficient conditions for explosivity and non-explosivity.

\thmb\label{th-7}
Assume $\rm(\mathbf{A1})$-$\rm(\mathbf{A5})$, and that $\cT$ is finite. Then, $Y_t$ is explosive with positive probability if and only if {either (C1) or (C2) holds.}
Moreover, $Y_t$ is explosive a.s. whenever it is explosive, provided $\partial=\varnothing$.
\thme

\begin{theorem}
  \label{cor-infinite-explosivity}
  Assume $\rm(\mathbf{A1})$-$\rm(\mathbf{A5})$ and that $\cT$ is infinite. Then, $Y_t$ is explosive if {(C1)} holds, and it is non-explosive if {either of the three conditions (C3), (C4), (C5) holds.} Moreover, $Y_t$ is explosive a.s.~whenever it is explosive, provided $\partial=\varnothing$.
\end{theorem}

{Explosion might occur  with  probability less than one for CTMCs with \emph{non}-polynomial transition rates and $\partial=\varnothing$ \cite{MP14}.}
Reuter's criterion and  { generalizations  of it} provide necessary and sufficient conditions for explosivity (with positive probability) for general CTMCs  in terms of convergence or divergence of a series \cite{CPZC05,LL09,R57}. {However, these conditions} are {\em not} easy to check.
In comparison, for CTMCs with polynomial transition rates, Theorem \ref{th-7} provides an {\em explicit} and checkable necessary and sufficient condition.

\subsection{Recurrence \emph{vs} transience, and certain absorption}

For a non-empty subset $A\subseteq\cX$, let $\tau_A=\inf\{t\ge0\colon Y_t\in A\}$ be the {\em hitting time} of $A$, with the convention that $\inf\varnothing=\infty$. If $Y_0\in A$, then $\tau_A=0$.  Let $\tau^+_A=\inf\{t\ge J_1\colon Y_t\in A\}$ be the {\em first return time} to $A$. Obviously, $\tau_A=\tau^+_A$ if and only if $Y_0\notin A$. The process $Y_t$  has \emph{certain absorption} if the hitting time of $\partial$ is finite a.s. for all $Y_0\in\cX$.

 \thmb\label{th-8}
Assume $\rm(\mathbf{A1})$-$\rm(\mathbf{A5})$ and that $\Omega$ is finite.

(i) Assume $\partial=\varnothing$. Then, $Y_t$ is recurrent if {either (C3) or (C6) holds},
while it is transient {if neither of them hold.}

(ii) Assume $\partial\neq\varnothing$. Then, $Y_t$ has certain absorption if and only if {either (C3) or (C6) holds.}

\thme

{The results  show that CTMCs with polynomial transition rates cannot have an infinite series of critical transitions from  recurrence to transience, for varying parameter values. This is contrary to CTMCs with non-polynomial transition rates; as discovered in \cite{MP14}. One might hope, that this phenomenon carries over to CTMCs with  polynomial transition rates in dimensions higher  than one.}

\begin{theorem}
  \label{cor-infinite-recurrence}
Assume $\rm(\mathbf{A1})$-$\rm(\mathbf{A5})$ and that $\Omega$ is infinite.

(i) Assume $\partial=\varnothing$. Then, $Y_t$ is recurrent if {(C3) holds, and it is transient if (C7) holds.}

(ii) Assume $\partial\neq\varnothing$. Then, $Y_t$ has certain absorption if {(C3) holds, while it has not certain absorption if (C7) holds.}
\end{theorem}

\subsection{Moments of hitting times}

Below we present threshold results on the existence of  moments of hitting times  {for recurrent states only, as transient states have infinite return time.} Therefore, in the light of Theorem~\ref{th-8}, we investigate the existence and non-existence of moments of hitting times only for $\alpha<0$ and for $\alpha=0,\ \beta\le0$. Moreover, limited by the tools we apply,  we do not discuss existence and non-existence of   moments of    absorption times for $\partial\neq\varnothing$.  Hence, we assume $Y_t$ is irreducible on  $\cX$ (equivalently, $\partial=\varnothing$) and provide existence and non-existence of moments of hitting times for states in $\cX$.

\thmb\label{th-11}
Assume $\rm(\mathbf{A1})$-$\rm(\mathbf{A5})$, $\partial=\varnothing$, and that $\Omega$ is finite. {Then,  the following holds:}

(i) There exists a finite non-empty subset $B\subseteq\cX$, such that
\eqb\label{Eq-exist}\E_x(\tau^{\epsilon}_B)<+\infty,\quad \forall x\in\cX,\quad {\forall 0<\epsilon<\delta,}
\eqe
{for $\delta>0$, provided  one of the conditions (C3), (C9), (C10) holds; for $0<\delta<1/2$, provided (C13) holds; for $0<\delta<\tfrac\beta{\beta-\gamma}$, provided (C14) holds; and for $0<\delta<1$, provided (C16)  holds.}
 {In particular, $\E_x(\tau_B)<+\infty$, provided one of the conditions (C3), (C9), (C10), (C11) holds.}

(ii)  There exists a finite non-empty subset $B\subseteq\cX$, such that
\begin{equation*}
\E_x(\tau_B^{\epsilon})=+\infty,\quad \forall x\in\cX\setminus B,\quad {\forall \epsilon>\delta,}
\end{equation*}
{for $\delta>1$, provided (C13) holds; for $\delta>\tfrac\beta{\beta-\gamma}$, provided  (C15)   holds; and for $\delta>1$, provided (C16)  holds.
 In particular, $\E_x(\tau_B)=+\infty$ provided  (C12) holds.}  
\thme

\begin{theorem}\label{cor-infinite-passagetime}
Assume $\rm(\mathbf{A1})$-$\rm(\mathbf{A5})$, $\partial=\varnothing$, and that $\Omega$ is infinite. {If (C3) is fulfilled and $0<\delta\le1$, then \eqref{Eq-exist} holds.}
\end{theorem}

\subsection{Positive recurrence and null recurrence}

We provide sharp criteria for positive and null recurrence, as well as  exponential ergodicity of stationary distributions and QSDs.

 If $\partial=\varnothing$, then $\tau_{\partial}=\infty$ a.s., and the conditional process $(Y_t\colon \tau_{\partial}>t)$ reduces to $(Y_t\colon t\ge 0)$.
If $\partial\neq\varnothing$, and $\tau_{\partial}<\infty$ a.s. (that is, $Y_t$ has certain absorption), then {process conditioned to never be absorbed} $(Y_t\colon \tau_{\partial}>t)$ is {referred to as the \emph{$Q$-process} \cite{CMM13,CV16}.} 

{The process $(Y_t\colon \tau_{\partial}>t)$} on $\cX\setminus\partial$
 is said to be \emph{exponentially ergodic},  if there exists a probability measure $\mu_*$  and $0<\delta<1$, such that for all probability measures $\mu$ on $\cX\setminus\partial$,
there exists a constant $C_{\mu}>0$, such that
\[{|\mathbb{P}_{\mu}(Y_t\in B|\tau_{\partial}>t)-\mu_*(B)|\le C_{\mu}\delta^t,\quad \forall t>0,\ B\subseteq \cX\setminus\partial}\]
 \cite{GS02}. {The measure $\mu^*$ is also said to be exponentially ergodic.}
In particular, if $C_{\mu}$ can be chosen independently of $\mu$, then $(Y_t\colon \tau_{\partial}>t)$ {and $\mu^*$} is said to \emph{uniformly exponentially ergodic}. {Moreover, when $\partial=\varnothing$, then $\mu_*$ is the unique ergodic stationary distribution; when $\partial\neq\varnothing$, then $\mu_*$ is a \emph{quasi-limiting distribution} (QLD)  
\cite{CMM13}. 
}

If $\partial\neq\varnothing$, a probability measure $\nu$ on $\cX\setminus\partial$ is a QSD for $Y_t$
 if for all $t\ge0$ and all  sets $B\subseteq \cX\setminus\partial$,
\[\mathbb{P}_{\nu}(Y_t\in B|\tau_{\partial}>t)=\nu(B).\]
{Any QLD is a QSD \cite{CMM13}.} The existence of a QSD implies certain absorption,  {and exponential ergodicity of the $Q$-process  
 implies  existence of a unique QSD \cite{CMM13}.} A probability measure $\nu$ on $\cX\setminus\partial$ is a {\em quasi-ergodic distribution}  if, for any $x\in\cX\setminus\partial$ and any bounded function $f$ on $\cX\setminus\partial$ \cite{BR99,HZZ19}, {the following limit holds:}
\[\lim_{t\to\infty}\E_x\!\!\lt(\frac{1}{t}\int_0^tf(Y_s){\rm d}s\Big|\tau_{\partial}>t\rt)=\int_{\cX\setminus\partial}f{\rm d}\nu.\]
{A quasi-ergodic distribution is in general different from a QSD \cite{HZZ19}.}
\thmb\label{th-9}
Assume $\rm(\mathbf{A1})$-$\rm(\mathbf{A5})$ and that $\Omega$ is finite.

(i) Assume $\partial=\varnothing$ and that $Y_t$ is recurrent. Then, $Y_t$ is positive recurrent and there exists a unique stationary distribution $\pi$ on $\cX$, if and only if one of the conditions {(C3), (C9), (C10), (C11)} holds, while $Y_t$ is null recurrent if and only if none of the conditions  {(C3), (C9), (C10), (C11)}  hold.
{Moreover, $Y_t$ is exponentially ergodic if either (C17) or (C18)  holds.}

(ii) Assume $\partial\neq\varnothing$ and that $Y_t$ has certain absorption.
Then, there exist no QSDs if none of the conditions {(C3), (C9), (C10), (C11)} hold. In contrast, there exists a unique uniformly exponentially ergodic {QLD},  supported on $\cX\setminus\partial$,  if {either (C18) or (C19)}  holds.
{Morevover, it is also a unique quasi-ergodic distribution and the unique stationary distribution of the $Q$-process.}
\thme

\begin{theorem}\label{cor-infinite-ergodicity}
Assume $\rm(\mathbf{A1})$-$\rm(\mathbf{A5})$ and that  $\cT$ is infinite.
 \begin{enumerate}
   \item[(i)] Assume $\partial=\varnothing$. Then, $Y_t$ is positive recurrent and there exists a unique stationary distribution $\pi$ on $\cX$, if {(C3)} holds. Moreover, $\pi$ is exponentially ergodic, if {(C17)} holds.
\item[(ii)] Assume $\partial\neq\varnothing$. Then, there exist no QSDs, if {(C7)} holds, while there exists a unique uniformly exponentially ergodic {QLD} 
supported on $\cX\setminus\partial$, if {(C19)}  holds.
 \end{enumerate}
\end{theorem}
{We provide some perspectives.}

$\bullet$ The convergence (or ergodicity) in Theorem \ref{cor-infinite-ergodicity}(ii) is uniform with respect to the initial distribution, while in contrast, the convergence  in Theorem \ref{cor-infinite-ergodicity}(i) is not uniform.  Indeed, for the subcritical linear BDP, the stationary distribution is exponentially ergodic but not uniformly so \cite{A91}.

$\bullet$ Indeed, one can obtain  {\em uniform} exponential ergodicity in Theorem \ref{th-9}(i) with {(C18) or (C19)  by choosing a non-reachable absorption set (potentially empty), hence imposing that the time to extinction is infinite. In so, the QSD is in fact a stationary distribution \cite{CV17}.}

 $\bullet$ The subtle difference between the conditions for positive recurrence and   for {exponential} ergodicity of QSDs lies in the fact that we have no {\em a priori} estimate of the {\em decay parameter}
\[\psi_0=\inf\bigl\{\psi>0\colon \liminf_{t\to\infty}e^{\psi t}\mathbb{P}_x(X_t=x)>0\bigr\} \]
(which is independent of $x$) \cite{CV17}.
We cannot compare $\psi_0$ with $-\alpha$ when $R>1$, or with $-\beta$, when $R>2$ and $\alpha=0$. Refer to the constructive proofs (using Lyapunov functions) in Appendix \ref{appB} for details. { Hence, one may believe that the condition we provide for quasi-ergodicity   generically is stronger than that for ergodicity.}

The only gap cases that remain for QSDs are {(C11), (C20), (C21), where neither existence of a QSD nor  exponential ergodicity of the $Q$-process are known to occur, provided one of the three conditions hold.}

\subsection{Implosivity}\label{subsec-implosive}

Assume $\partial=\varnothing$. Then, $Y_t$ is irreducible {on $\cX$.} Let $B\subsetneq\cX$ be a non-empty proper subset. Then $Y_t$ {\em implodes towards $B$}
\cite{MP14} if there exists $t_*>0$, such that  
\[\E_y(\tau_B)\le t_*,\quad \forall y\in \cX \setminus B.\]
Implosion towards a single state $x\in\cX$ implies finite expected first return time {to the state}, and thus positive recurrence of $x$. Indeed,
\[\E_x(\tau^+_x)\le \E_x(J_1)+\sup\nolimits_{\{y\colon y\neq x\}}\E_y(\tau_x)<\infty,\]
where $\tau_x^+=\tau_{\{x\}}^+$, $\tau_x=\tau_{\{x\}}$, and $J_1$ has finite expectation since $x$ is not absorbing.
Hence, $Y_t$ does not implode towards any transient or null recurrent state.

The process $Y_t$ is {\em implosive} if $Y_t$ implodes towards any state of $\cX$, and otherwise, $Y_t$ is {\em non-implosive}. Hence, implosivity implies positive recurrence.
If $Y_t$ implodes towards a finite non-empty subset of $\cX$, then $Y_t$ is implosive (see Proposition~\ref{Spro-10}).

\thmb\label{th-10}
Assume $\rm(\mathbf{A1})$-$\rm(\mathbf{A5})$, $\partial=\varnothing$, and that
$\Omega$ is finite. Then, $Y_t$ is implosive, and there exists $\epsilon>0$, such that for every non-empty finite subset $B\subseteq\cX$ and every $x\in \cX\setminus B$,
$$\E_x(\exp\lt(\tau_B^{\epsilon}\rt))<\infty,$$
if either {(C18) or (C19) holds,} while $Y_t$ is non-implosive otherwise.
\thme

\begin{theorem}
  \label{cor-infinite-implosivity}
Assume $\rm(\mathbf{A1})$-$\rm(\mathbf{A5})$, $\partial=\varnothing$,
 and that $\cT$ is infinite. Then, $Y_t$ is implosive if {(C19)  holds.}
\end{theorem}

\subsection{{Relations between absorbed CTMCs and non-absorbed CTMCs}}
{It is worth comparing the properties of  absorbed CTMCs to those of  non-absorbed CTMCs, as first discussed by Karlin and McGregor \cite{KM57}. Indeed, in the case of a finite set of absorbing states $\partial\neq\varnothing$ and an irreducible state space $\cX\setminus \partial$,  one can add positive transition rates  to the transition matrix from the states in $\partial$ to a finite subset of states in $\cX\setminus\partial$, such that $\cX$ is irreducible for the new chain. Conversely, if $\partial=\varnothing$, then one can prescribe a finite set $\partial'\subseteq\cX$, and delete all transitions from $\partial'$ to $\cX'=\cX\setminus\partial$, such  that $\cX'$ is irreducible and $\partial'$ an absorbing set for the new chain. {These operations  can be viewed as simple extensions to those of \cite{KM57}, proposed in the context of BDPs}. As the dynamical properties we discuss  generically are determined by transitions among states of large values, the  operations provide a way to link the  dynamics of an absorbed CTMC with that of a corresponding non-absorbed CTMC, and  \emph{vice versa}.}

{ Figure \ref{fig:implications} shows the implications among the properties, in agreement with the parameter conditions derived in the main theorems. In Examples ~\ref{ex-QSD} and \ref{Ex-implosivity} below, counterexamples are given.
It remains unknown  whether {exponential ergodicity of  the $Q$-process implies implosivity, and whether ergodicity implies existence of QSD, see  Figure \ref{fig:implications}.}}

\newcommand*{\nNWarrow}{\rotatebox[origin=c]{150}{\(\nRightarrow\)}}

\color{black}
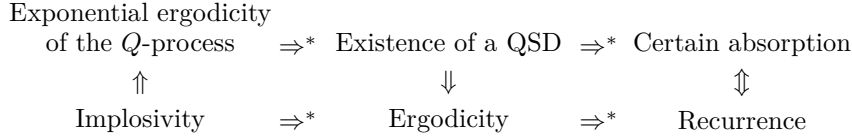
\begin{figure}\label{fig:implications}
\begin{center}
\begin{tikzpicture}[scale=0.5]
\node   at (0,.8) {Exponential ergodicity};
\node   at (0,0) {of the $Q$-process};
\node   at (0,-2) {Implosivity};
\node   at (0,-1) {$\Uparrow$};
\node   at (4.2,0) {$\Rightarrow^*$};
\node   at (4.2,-2) {$\Rightarrow^*$};
\node   at (8.2,-2) {Ergodicity};
\node   at (8.2,0) {Existence of a QSD};
\node   at (8.2,-1) {$\Downarrow$};
\node   at (12.2,0) {$\Rightarrow^*$};
\node   at (12.2,-2) {$\Rightarrow^*$};
\node   at (16,-2) {Recurrence};
\node   at (16,0) {Certain absorption};
\node   at (16,-1) {$\Updownarrow$};

\end{tikzpicture}
\end{center}
\caption{{Implications among properties. Ergodicity and recurrence refer to the non-absorped process; exponential ergodicity of the $Q$-process, existence of a QSD and certain absorption to the absorped process; and implositity to both processes. Implications that cannot be extended to biimplications are marked by an asterix $^*$. These are justified by  counterexamples in Examples~\ref{ex-QSD} and \ref{Ex-implosivity}, except for the implication Ergodicity $\Rightarrow$ Recurrence. Its  reverse implication fails according to Theorem~\ref{th-9}(i). Ergodicity does not imply exponential ergodicity of the {Q-process}, see also  Examples~\ref{ex-QSD}(i).}}\label{fig:implications}
\end{figure}
\color{black}

\eb\label{ex-QSD}
(i) Consider the sublinear {BDP on $\N_0$ } with birth rates $\lambda_j=a$ and death rates $\mu_j=b$ for $j\in\N$. We have $\partial=\{0\}$, $R=0$, and $\alpha=a-b$. Hence, the process is non-explosive  for any initial state by Theorem~\ref{th-7}. By \cite{V91}, the process has certain absorption with  decay parameter $\psi_0=(\sqrt{a}-\sqrt{b})^2$, when $a\le b$, and it admits a continuum family of QSDs, when $\alpha<0$. {This shows that existence of QSD does not imply exponential ergodicity of  {the Q-process} in general.}

(ii) Consider the linear {BDP on $\N_0$} with birth rates $\lambda_j=aj$ and death rates $\mu_j=bj$ for $j\in\N$. Assume $a\le b$. We have,  $\partial=\{0\}$, $R=1$, and $\alpha=a-b$. Hence, the process is non-explosive for any initial state by Theorem~\ref{th-7}. By \cite{V91}, the process has certain absorption with decay parameter $\psi_0=(\sqrt{a}-\sqrt{b})^2$, when $a\le b$, and it admits no QSDs for $\alpha=0$ {(hence, $\beta=-a<0=\gamma$)}, while a continuum family of QSDs for $\alpha<0$. {This shows that the process has certain absorption, but no QSDs for $\alpha=0$, which is also justified by Theorem~\ref{th-9}(ii). Moreover, it also shows that ergocidity of the non-absorbed process does not imply   uniqueness of a QSD  or  exponential  ergodicity of the $Q$-process.}

(iii) Consider the superlinear {BDP on $\N_0$} with birth rates $\lambda_j=j^2$ and death rates $\mu_j=j^2$ for $j\in\N$. We have $\partial=\{0\}$, $R=2$, $\alpha=0$, and $\beta=-1<0$. {Hence, the process is non-explosive and has certain absorption by Theorem~\ref{th-7} and Theorem~\ref{th-8}(ii). By \cite{V91}, the process admits either no QSDs or a continuum family of QSDs. This shows that certain absorption does not imply exponential ergodicity of  the Q-process}. 
\ee

Implosivity is  indeed {a stronger property} than positive recurrence (e.g., when $R\le1$, $\alpha<0$), as shown in the following example (see also Table~1).

\eb\label{Ex-implosivity} Let $Y_t$ be an irreducible BDP on $\N_0$ with $\cT=\{1,-1\}$ and
$$\lambda_{-1}(x)=x,\quad \lambda_1(x)=1,\quad x\in\N_0.$$
In this case, $R=1$ and $\alpha=-1$. By Theorems~\ref{th-9} and \ref{th-10}, $Y_t$ is positive recurrent and admits an ergodic stationary distribution, but $Y_t$ is non-implosive.
\ee

\section{Applications}\label{sec5}

\subsection{Stochastic reaction networks}\label{sec:RN}
{A \emph{reaction network} $(\mathcal{C},\cR)$ on a finite set $\cS=\{S_1,\ldots,S_n\}$ (with elements called   \emph{species}) is an edge-labelled finite digraph with node set $\mathcal{C}$ (with elements called \emph{complexes}) and edge set $\cR$ (with elements called \emph{reactions}), such that the elements of $\mathcal{C}$ are non-negative linear combinations of species, $y=\sum_{i=1}^ny^i S_i$,  identified with vectors $y=(y^1,\ldots,y^n)$ in $\N_0^n$.  Reactions are directed edges between complexes,  written as $y\to y'$. We assume every species has a positive coefficient in some complex, and that every complex is in some reaction. Hence, the reaction network might be deduced from the reactions alone and it is custom just to list (or draw) the reactions. If $n=1$, the reaction network is a one-species reaction network.}

{A {\em stochastic reaction network} (SRN)  is a reaction network together with a  CTMC $X(t), t\ge 0$, on {$\N_0^n$}, modeling the number of molecules of each species over time. A reaction $  y\to y'$ \emph{fires}  with transition rate $\eta_{y\to y'}(x)$,
in which case the chain jumps from $X(t)=x$ to $x+y'-y$ \cite{AK15}.
The Markov process with transition rates $\eta_{y\to y'}\colon\N^n_0\to \R_{\geq 0}$, ${y\to y'}\in\cR$, has  $Q$-matrix
$$q_{x,x+\omega}=\sum_{y \to y'\in \cR\colon y'-y=\omega}\eta_{y\to y'}(x).$$
Hence, the transition rate from the state $x$ to $x+\omega$ is 
 $$\lambda_{\omega}(x)=\sum_{y \to y'\in \cR\colon y'-y=\omega}\eta_r(x).$$
For (stochastic) mass-action kinetics, the transition rate for $y\to y'$ is
\begin{equation*}
\label{int}
\eta_{y\to y'}(x)=\ka_{y\to y'}\tfrac{(x)!}{(x-y)!}1_{\{x'\colon x'\geq y\}}(x),\quad x\in\N^n_0 
\end{equation*}
(in accordance with the transition rate introduced in the Introduction for one-species reaction networks), where $x!:=\prod_{i=1}^nx_i!$,
and  $\ka_{y\to y'}$ is a positive reaction rate constant  \cite{AK11,AK15}. Generally, we number the reactions and write $\ka_1,\ka_2,\ldots$ for convenience.
}

{In this section, we apply the results developed in Section~\ref{sect-criteria} to some examples of SRNs.}

\eb Consider the following two reaction networks:
{\[(A)\quad \varnothing \ce{<=>[\ka^A_1][\ka^A_2]} \tS,\qquad \text{and}\qquad(B)\quad \varnothing \ce{<=>[\ka^B_1][\ka^B_2]} \tS,\quad 2\tS \ce{->[\ka^B_3]}3\tS,\]}
with $\Omega=\{-1,1\}$ in both cases and with transition rates
{\begin{align*}
\lambda^A_{-1}(x)&=\ka^A_2 x,\qquad \lambda^A_{1}(x) =\ka^A_1,\\
\lambda^B_{-1}(x)&=\ka^B_2 x,\qquad \lambda^B_{1}(x)=\ka^B_1+\ka^B_3x(x-1),
\end{align*}}
respectively.
 By Theorem~\ref{th-9}, the first is positive recurrent and admits an exponentially ergodic stationary distribution on $\N_0$ since {$\alpha=-\ka^A_2$} and $R=1$, while by  Theorem~\ref{th-7}, the second reaction network is explosive for any initial state since {$\alpha=\ka^B_3>0$} and $R=2$. Indeed, these two reaction networks are \emph{structurally equivalent} in the sense that there is only one irreducible component $\N_0$  \cite{XHW20a}.
\ee

\eb
Consider the following pair of SRNs from the introduction:
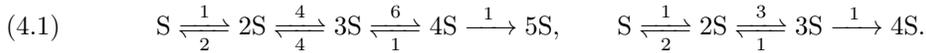
\eqb\label{Eq-1}
\text{S}\ce{<=>[1][2]} 2\text{S}\ce{<=>[4][4]} 3\text{S}\ce{<=>[6][1]} 4\text{S}\ce{->[1]} 5\text{S},\qquad \text{S}\ce{<=>[1][2]}  2\text{S}\ce{<=>[3][1]}  3\text{S}\ce{->[1]}  4\text{S}.
\eqe
For the first reaction network, $R=4$, $\alpha=0$ and $\beta=1$, and for the second, $R=3$, $\alpha=0$ and $\beta=0$. By Theorem~\ref{th-7}, the first is explosive  for any initial state and the second does not explode for any initial state.
\ee

\eb
{\rm(i)} Consider a strongly connected reaction network:
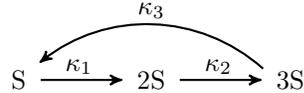
\[
\begin{tikzpicture}[node distance=3.5em, auto, scale=1]
 \tikzset{
    >=stealth',
    pil/.style={
           ->,
           thick,
           shorten <=2pt,
           shorten >=2pt,}
}
 \node[] (a) {};
  \node[right=1.3em of a] (n1) {};
  \node[above=-.5em of n1] (m1) {$\ka_1$};
\node[right=of a] (b) {2S};
\node[right=1.3em of b] (n2) {};
  \node[above=-.5em of n2] (m2) {$\ka_2$};
  \node[above=1.3em of b] (m3) {$\ka_3$};
  \node[left=of b] (aa) {S} edge[pil, black, bend left=0] (b);
\node[right=of b] (c) {3S};
 \node[left=of c] (bb) {} edge[pil, black, bend left=0] (c);
\node[right=of b] (cc) {} edge[pil, black, bend right=40] (aa);
\end{tikzpicture}\]
For the underlying CTMC $Y_t$, $\cT=\{1,-2\}$, and \[\lambda_1(x) =\ka_1 x+\ka_2x(x-1),\quad \lambda_{-2}(x)=\ka_3 x(x-1)(x-2).\]
 Hence $Y_t$ is irreducible on $\N$ with $0$ a neutral state. 
Moreover, $R=3$, and $\alpha=-2\ka_3$. By Theorem~\ref{th-9}, there exists a unique exponentially ergodic stationary distribution on $\N$.

{\rm(ii)} Consider a similar reaction network including direct degradation of S:
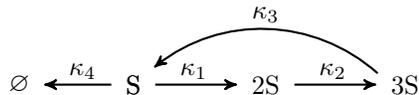
\[
\begin{tikzpicture}[node distance=3.5em, auto, scale=1]
 \tikzset{
    >=stealth',
    pil/.style={
           ->,
           thick,
           shorten <=2pt,
           shorten >=2pt,}
}
\node[] (a) {};
\node[right=1.3em of a] (n1) {};
\node[left=of a] (d) {$\varnothing$};
\node[right=1.3em of d] (n4) {};
\node[above=-.5em of n4] (m4) {$\ka_{4}$};
\node[right=1.3em of b] (n2) {};
\node[above=-.5em of n1] (m1) {$\ka_1$};
\node[right=of a] (b) {2S};
\node[right=1.3em of b] (n2) {};
\node[above=-.5em of n2] (m2) {$\ka_2$};
\node[above=1.3em of b] (m3) {$\ka_3$}; \node[left=of b] (aa) {S} edge[pil, black, bend left=0] (d);
  \node[left=of b] (aa) {S} edge[pil, black, bend left=0] (b);
\node[right=of b] (c) {3S};
 \node[left=of c] (bb) {} edge[pil, black, bend left=0] (c);
\node[right=of b] (cc) {} edge[pil, black, bend right=40] (aa);
\end{tikzpicture}\]
The threshold parameters are the same as in {\rm(i)}. Let $\partial=\{0\}$ with $\N_0\setminus\partial=\N$
an irreducible component. By Theorems~\ref{th-9}, the network has a uniformly exponentially ergodic QSD on $\N$.
\ee

\subsection{An extended class of branching processes}

Consider an extended class of branching processes \cite{C97} with transition rate matrix $Q=(q_{x,y})_{x,y\in\N_0}$:
\eqb\label{Eq-5}q_{x,y}=\left\{\begin{array}{cl}
r(x)\mu(y-x+1), &\quad  \text{if}\quad  y\ge x-1\ge0\quad \text{and}\quad y\neq x,\\
 -r(x)(1-\mu(1)), &\quad  \text{if}\quad y=x\ge1,\\
 q_{0,y}, & \quad \text{if}\quad y>x=0,\\
 -q_0, &\quad \text{if}\quad y=x=0,\\
 0, &\quad  \text{otherwise},
\end{array}\right.\eqe
where $\mu$ is a probability measure on $\N_0$, $q_0=\sum_{y\in\N}q_{0,y}$, and $r(x)$ is a positive finite function on $\N_0$.
Assume

\medskip
\noindent($\rm\mathbf{H1}$) $\mu(0)>0$, $\mu(0)+\mu(1)<1$.

\medskip
\noindent($\rm\mathbf{H2}$) $\sum_{y\in\N}q_{0,y}y<\infty$, ${E}
=\sum_{k\in\N_0}k\mu(k)<\infty$.
\medskip

\noindent($\rm\mathbf{H3}$) $r(x)$ is a polynomial of degree $R\ge1$ for large $x$.
\medskip

The next theorem follows from the results in Section~\ref{sect-criteria}. We would like to mention that the results below provide conditions for different dynamical regimes in terms of only $R$ and $M$. In contrast, the condition for positive recurrence in \cite{C97} also depends on the integrability of a definite integral  as well
as summability of a series, which nonetheless never appear.
\thmb\label{th-branching}
Assume ${\rm(\mathbf{H1})}$-${\rm(\mathbf{H3})}$. Let $Y_t$ be a process generated by the $Q$-matrix given in \eqref{Eq-5} and $Y_0\neq0$. Then $Y_t$ is non-explosive if  one of the following conditions holds: (1) $R\le1$, (2) ${E}<1$, (3) $R=2$, ${E}=1$, while explosive with positive probability if (4) ${E}>1$, $R>1$. Furthermore,
\enb
\item[(i)] if $q_0>0$, then $Y_t$ is irreducible on $\N_0$ and is
\enb
\item[(i-1)] recurrent if ${E}<1$, and transient if ${E}>1$.
\item[(i-2)] positive recurrent and exponentially ergodic if ${E}<1$.
\item[(i-3)] implosive if $R>1$ and ${E}<1$.
\ene
\item[(ii)] if $q_0=0$, then $\partial=\{0\}$, and $Y_t$ has certain absorption if ${E}<1$, while it has not if ${E}>1$. Moreover, the process admits no QSDs if ${E}>1$, while it admits a uniformly exponentially ergodic QSD on $\N$ if $R>1$ and ${E}<1$.
\ene
\thme

\prb
For all $k\in\N\cup\{-1\}$, let
$$\lambda_k(x)=\cab r(x)\mu(k+1),\quad \text{if}\ x\in\N,\\ q_{0k},\qquad\qquad\quad\, \text{if}\ x=0.\cae$$
 By ${\rm(\mathbf{H1})}$, $\mu(k)>0$ for some $k\in\N$. Note that $\cX\setminus\partial$ is irreducible, with $\partial=\varnothing$ if $q_0>0$ and  $\partial=\{0\}$ if $q_0=0$. 
Hence regardless of $q_0$, by positivity of $r$, ${\rm(\mathbf{A1})}$-${\rm(\mathbf{A2})}$ are  satisfied with $\cT_-=\{-1\}$ and $\cT_+=\{y\colon q_{0y}>0\}\cup(\supp\mu\setminus\{0,1\}-1)$. Moreover, ${\rm(\mathbf{H2})}$-${\rm(\mathbf{H3})}$ imply ${\rm(\mathbf{A3})}$-${\rm(\mathbf{A4})}$. Let $r(x)=ax^R+bx^{R-1}+\rO(x^{R-2})$ with $a>0$. Since $R\ge1$, it  coincides with $\max\{\deg(\lambda_{\omega})\colon \omega\in\cT\}$. It is straightforward to verify that
$$\alpha=a({E}-1),\quad \beta=(\tfrac{1}{2}a+b)({E}-1)-\tfrac{1}{2}a{E'}, \text{and}\quad \gamma=b({E}-1),$$
where  ${E'}=\sum_{k\in\N}k(k-1)\mu(k)>0$. Hence $\alpha$ has the same sign as ${E}-1$, and $\beta<0$ whenever ${E}=1$ (or equivalently, $\alpha=0$). Furthermore, $\alpha=0$ implies $\gamma=0$.  In addition, in the light of $R\ge1$, the condition ${E}\ge 1$ decomposes into three possibilities:
$${E}>1,\quad \text{or}\quad {E}=1, R>1,\quad \text{or}\quad {E}=R=1.$$
  Then the conclusions follow directly from Theorems~\ref{th-7}, \ref{th-8}, \ref{th-9} and \ref{th-10}.
\pre

\begin{corollary}
Assume ${\rm(\mathbf{H1})}$-${\rm(\mathbf{H3})}$ and $\mu$ has finite support and $\{y\colon q_{0y}>0\}$ is finite. Then  $Y_t$  is non-explosive if and only if either $R=1$ or ${E}\le1$. Furthermore, \enb
\item[(i)] if $q_0>0$, then $Y_t$ is irreducible and is
\enb
\item[(i-1)] recurrent if ${E}\le1$, and transient otherwise.
\item[(i-2)] positive recurrent if and only if ${E}<1$, or ${E}=1$, $R>1$, while null recurrent if and only if ${E}=R=1$. Furthermore, $Y_t$ is exponentially ergodic if ${E}<1$, or ${E}=1$, $R>2$.
\item[(i-3)] implosive if and only if either of the two conditions  $R>1$, ${E}<1$, or $R>2$, ${E}=1$ holds.
\ene
\item[(ii)] if $q_0=0$, then $\partial=\{0\}$, and $Y_t$ has certain absorption if and only if ${E}\le1$. Moreover, the process admits no QSDs if ${E}>1$, or ${E}=R=1$, while it admits a uniformly exponentially ergodic QSD on $\N$ if either $R>1$, ${E}<1$, or $R>2$, ${E}=1$.
\ene
\end{corollary}
\begin{proof}
  Based on the proof of Theorem~\ref{th-branching}, the conclusions follow from Corollaries~\ref{cor-infinite-explosivity}, \ref{cor-infinite-recurrence}, \ref{cor-infinite-ergodicity} and \ref{cor-infinite-implosivity}.
\end{proof}

The extended branching process under more general assumptions (allowing more general forms of $r$) is addressed in \cite{C97}. In that reference, the conditions given for the dynamic behavior of the process seem more involved than here and even become void in some situations (e.g., in \cite[Corollary~1.5(iii)]{C97}, where the definite integral  indeed is always  infinite  under ($\rm\mathbf{H1}$)-($\rm\mathbf{H3}$).)

\subsection{A general single-cell stochastic gene expression model}\label{subsec-gene}

To model single-cell stochastic gene expression with bursty production, we propose the following one-species \emph{generalized} reaction network (consisting potentially of infinitely many reactions) with mass-action kinetics:
\begin{equation}\label{Eq-3}
m\tS\ce{->[c_m\mu_m(k)]}(m+k)\tS,\quad m=0,\ldots,J_1,\qquad m\tS\ce{->[r_m]}(m-1)\tS,\quad m=1,\ldots,J_2,\end{equation}
where $c_m\ge0$ for $m=0,\ldots,J_1$, $r_m\ge0$ for $m=1,\ldots,J_2$, $J_1\in\N_0,\ J_2\in\N$, and  $\mu_m$, for $m=0,\ldots,J_1$, are probability distributions on $\N$. Assume

\medskip
\noindent ${\rm(\mathbf{H4})}$ $J_1\le J_2$, $c_0>0$, $c_{J_1}>0$, $r_1>0$, and $r_{J_2}>0$.
\medskip

\noindent ${\rm(\mathbf{H5})}$
${E}_m=\sum_{k=1}^{\infty}k\mu_m(k)<\infty$, for $m=0,\ldots,J_1$.
\medskip

This network embraces several single-cell stochastic gene expression models in the presence of bursting, see e.g. \cite{BA16,CJ20,MTY13,SS08}. Ergodicity as well as an  exact formula for the ergodic stationary distribution (when it exists) are the main concerns of these references. The first set of $J_1$ reactions
 account for  bursty production of mRNA copies with transcription rate $c_m$ and burst size distribution $\mu_m$.
The second set of $J_2$ reactions
  account for degradation of  mRNA with degradation rate $r_m$ \cite{CJ20,SS08}.

 The network \eqref{Eq-3} reduces to the specific model studied in
\enb
\item[$\bullet$] \cite[Section~4]{CJ20} (see also \cite[Section~3.2]{MTY13}), when $J_1=0$, $J_2=1$, and $\mu_0$ is a geometric distribution.
\item[$\bullet$] \cite{SRB12}, when $J_1=0$, $J_2=1$, and $\mu_0$ is a negative binomial distribution.
\item[$\bullet$] \cite[Example~3.6]{MTY13}, when $J_1=J_2=1$, and $\mu_0=\mu_1$ are geometric distributions.
\item[$\bullet$] \cite{FMD17} when $J_1=2$, $J_2=3$, $\mu_0=\delta_1$, $\mu_2=\delta_k$ for some $k\in\N$, and $c_1=r_2=0$. Here $\delta_i$ is the Dirac delta measure at $i$.
\ene

\thmb\label{th-gene}
Assume ${\rm(\mathbf{H4})}$-${\rm(\mathbf{H5})}$, and that $\mu_m$ has finite support whenever $c_m>0$ for $m=0,\ldots,J_1$. Then the process $Y_t$ associated with the network \eqref{Eq-3} is irreducible on $\N_0$, and is positive recurrent and there exists an ergodic stationary distribution on $\N_0$ if and only if one of the following conditions holds:
\enb
\item[(i)] $J_1<J_2$,
\item[(ii)] $J_1=J_2$ and $c_{J_2}{E}_{J_2}<r_{J_2}$,
\item[(iii)] $J_1=J_2>2$, $c_{J_2}{E}_{J_2}=r_{J_2}$, and $c_{J_2-1}{E}_{J_2-1}\le r_{J_2-1}+\frac{1}{2}c_{J_2}({E}_{J_2}+{E}'_{J_2})$,
\item[(iv)] $J_1=J_2=2$, $c_{J_2}{E}_{J_2}=r_{J_2}$, and $c_{J_2-1}{E}_{J_2-1}<r_{J_2-1}+\frac{1}{2}c_{J_2}({E}_{J_2}+{E}'_{J_2})$,
\item[(v)]  $J_1=J_2=1$, $c_{J_2}{E}_{J_2}=r_{J_2}$, and $c_{J_2-1}{E}_{J_2-1}<r_{J_2-1}$,
\ene
where ${E}'_m=\sum_{k=1}^{\infty}k^2\mu_m(k)$.
Moreover, the stationary distribution is exponentially ergodic if one of (i), (ii) and (iii) holds. Besides, the process $Y_t$ is implosive if and only if (iii) or $J_2>1$ with (i) or (ii).
\thme
\prb
We have  $\cT=\{-1\}\cup(\cup_{j=0}^{J_1}\supp\mu_j)$, and
$$\lambda_{-1}(x)=\sum_{j=1}^{J_2}r_jx^{\underline{j}},\quad \lambda_k(x)=\sum_{j=0}^{J_1}c_j\mu_j(k)x^{\underline{j}},$$
 for $k\in\N$ and $x\in\N_0$, where $x^{\underline{j}}=\prod_{i=0}^{j-1}(x-i)$ is the descending factorial. By ${\rm(\mathbf{H4})}$, ${\rm(\mathbf{A1})}$-${\rm(\mathbf{A2})}$ are satisfied; moreover, the irreducibility of $Y_t$ also follows from \cite{XHW20a}. Under ${\rm(\mathbf{H5})}$, the mass-action kinetics yield  ${\rm(\mathbf{A3})}$-${\rm(\mathbf{A4})}$. Since $J_1\le J_2$ by ${\rm(\mathbf{H4})}$, we have $R=J_2\ge1$. Since
 $$\sum_{\om\in\cT}\lambda_{\om}(x)\om=-\sum_{j=J_1+1}^{J_2}r_jx^{\underline{j}}+\sum_{j=1}^{J_1}(c_j{E}_j-r_j)x^{\underline{j}}+c_0{E}_0,$$
  $$\sum_{\om\in\cT}\lambda_{\om}(x)\om^2=\sum_{j=J_1+1}^{J_2}r_jx^{\underline{j}}+\sum_{j=1}^{J_1}(c_j{E}'_j+r_j)x^{\underline{j}}
+c_0{E}'_0,$$
we have $\alpha=c_{J_1}{E}_{J_1}\delta_{J_1,J_2}-r_{J_2}$, where $\delta_{i,j}$ is the Kronecker delta. When $\alpha=0$, we have
 $$J_1=J_2,\quad c_{J_2}{E}_{J_2}=r_{J_2},\quad \gamma=c_{J_2-1}{E}_{J_2-1}-r_{J_2-1},$$
 $$ \beta=c_{J_2-1}{E}_{J_2-1}-r_{J_2-1}-\frac{1}{2}c_{J_2}({E}_{J_2}+{E}'_{J_2}).$$
Condition (i)+(ii) is equivalent to $\alpha<0$; condition (iii) is equivalent to $R>2$, $\alpha=0$, $\beta\le0$; condition (iv) is equivalent to $R=2$, $\alpha=0$, $\beta<0$; condition (v) is equivalent to $R=1$, $\alpha=0$, $\gamma<0$. Then the conclusions follow from Theorems~\ref{th-9} and \ref{th-10}.
\pre

\begin{corollary}
Assume ${\rm(\mathbf{H4})}$-${\rm(\mathbf{H5})}$, and  that $\mu_m$ has an infinite support and $c_m>0$ for some $m=0,\ldots,J_1$. Then $Y_t$ is irreducible on $\N_0$, and is positive recurrent with an ergodic stationary distribution if one of  Theorem~\ref{th-gene}(i), Theorem~\ref{th-gene}(ii) and Theorem~\ref{th-gene}(v) holds. Moreover, the stationary distribution is exponentially ergodic if either Theorem~\ref{th-gene}(i) or Theorem~\ref{th-gene}(ii) holds. In addition, the process $Y_t$ is implosive if $J_2>1$.
\end{corollary}
\begin{proof}
  Based on the proof of Theorem~\ref{th-gene}, the conclusions follow directly from Corollaries~\ref{cor-infinite-ergodicity} and \ref{cor-infinite-implosivity}.
\end{proof}
\subsection{Stochastic populations under bursty reproduction}
Two stochastic population models with bursty reproduction are investigated in \cite{BA16}.

The first model is a Verhulst logistic population process with bursty reproduction. The process $Y_t$ is a CTMC on $\N_0$ with transition rate matrix $Q=(q_{x,y})_{x,y\in\N_0}$ satisfying:
\[q_{x,y}=\cab c\mu(j)x,\qquad \text{if}\ y=x+j,\ j\in\N,\\ \frac{c}{K}x^2+x,\quad \text{if}\ y=x-1\in\N_0,\\ 0,\qquad\qquad\ \text{otherwise},\cae\]
where $c>0$ is the  reproduction rate, $K\in\N$ is the typical population size in the long-lived
metastable state prior to extinction \cite{BA16}, and $\mu$ is the burst size   distribution. Assume

\medskip
\noindent ${\rm(\mathbf{H6})}$ ${E_b}=\sum_{k=1}^{\infty}k\mu(k)<\infty$.
\medskip

Approximations of the mean time to extinction and QSD are discussed in \cite{BA16} against various different burst size distributions  of finite mean (e.g., Dirac measure, Poisson distribution, geometric distribution, negative-binomial distribution). Nevertheless, the existence of QSD is not proved there. Here we prove the certain absorption and ergodicity of the QSD for this population model.

\thmb
Assume ${\rm(\mathbf{H6})}$. The Verhulst logistic model $Y_t$ with bursty reproduction has certain absorption. Moreover, there exists a uniformly exponentially ergodic QSD on $\N$ trapped to zero.
\thme
\prb
We have $\cT=\supp\mu\cup\{-1\}$, $\lambda_{-1}(x)=\frac{c}{K}x^2+x$, $\lambda_k(x)=c\mu(k)x$, for $k\in\N$ and $x\in\N$. Let $\partial=\{0\}$, and $\N_0\setminus\partial=\N$ is irreducible \cite{XHW20a}. Hence ${\rm(\mathbf{A1})}$-${\rm(\mathbf{A5})}$ are satisfied. Moreover, $R=2$, $\alpha=-\frac{c}{K}<0$, and thus the conclusions follow from Theorems~\ref{th-8} and \ref{th-9}, and Corollaries~\ref{cor-infinite-recurrence} and \ref{cor-infinite-ergodicity} for finite $\supp\mu$ and infinite $\supp\mu$, respectively.
\pre

The second model is a runaway model of a stochastic population including bursty pair reproduction \cite{BA16}. This model can be described as a generalized reaction network:
where $c$, $K$ and $\mu$ are defined as in the first model. The survival probability of this population model is addressed in \cite{BA16}. Nevertheless, it turns out that this model is explosive for any initial state.

\thmb
Assume ${\rm(\mathbf{H6})}$. The runaway model is explosive.
\thme
\prb
We have $\cT=\supp\mu\cup\{-1\}$, $\lambda_k(x)=\frac{c}{K}\mu(k)x(x-1)$, $\lambda_{-1}(x)=x$, for $k\in\N$ and $x\in\N$. Let $\partial=\{0,1\}$. Then $\N_0\setminus\partial=\N\setminus\{1\}$ is irreducible \cite{XHW20a}. Hence ${\rm(\mathbf{A5})}$ is valid. Moreover, it is easy to verify that  ${\rm(\mathbf{A1})}$-${\rm(\mathbf{A4})}$ are also satisfied. In addition, $R=2$, $\alpha=\frac{c}{K}{E_b}>0$, and thus the conclusions follow from Theorem~\ref{th-7} and Corollary~\ref{cor-infinite-explosivity} for finite $\supp\mu$ and infinite $\supp\mu$, respectively.
\pre

\section{{Proofs}}\label{sec:proofs}

\subsection{{Proof of Theorem \ref{th-7}}}
Hereafter, we use the notation  $[m,n]_1$ ($[m,n[_1$, etc.) for the set of consecutive integers from $m$ to $n$, with $m,n\in\N_0\cup\{+\infty\}$. The notation is adopted from \cite{XHW20a}.

Moreover, throughout the proofs, we assume without loss of generality that $\cX=\N_0$, and $\partial\subseteq\{0\}$ for the ease of exposition. {Indeed,   the dynamical properties of the CTMCs discussed in this paper  depend  only on the transition structure of the states $x\in\cX$ with large value of $x$. By assumption ($\rm\mathbf{A4}$)-($\rm\mathbf{A5}$)   all  jumps in $\cT$ are possible.} When $\partial\neq\varnothing$, it is standard to `glue' all states in $\partial$ to be a single state $0$, since $\partial$ is finite and the set of states in $\cX\setminus\partial$ one jump away from $\partial$ is also finite, due to ($\rm\mathbf{A2}$).

\smallskip

We prove the conclusions case by case.

\noindent{\rm(a)} Assume $\partial=\varnothing$. Then, $Y_t$ is irreducible on $\N_0$ and can directly apply the Propositions~\ref{Spro-3} and \ref{Spro-13} with appropriate Lyapunov functions to be determined.

\noindent{\rm(b)} Assume $\partial\neq\varnothing$. Let $Z_t$ be the irreducible CTMC on the state space $\cX\setminus\partial$ with $Z_0=Y_0$ and transition operator $\widetilde{Q}$ being $Q$ restricted to $\cX\setminus\partial$: $$\widetilde{q}_{x,y}=q_{x,y},\ \text{for all}\ x,\ y\in\cX\setminus\partial \ \text{and}\ x\neq y.$$ In the following, we show that $Z_t$ is explosive if and only if $Y_t$ is explosive, and hence case (b) reduces to case (a).  {This equivalence is not quite trivial. There is positive probability that starting from any non-absorbing state the chain will jump to an absorbing state in a finite number of steps. So we need to show that this will not happen with probability one. Otherwise, explosivity is not possible. } Assume first that $Z_t$ is explosive. Then $\widetilde{Q}v=v$ for some bounded non-negative non-zero $v$. Let $u_x=v_x\mathbbm{1}_{\cX\setminus\partial}(x),\ \forall x\in\cX$. It is straightforward to verify that $Qu=u$. By Proposition~\ref{Spro-R}, $Y_t$ is also explosive. Conversely, assume that $Y_t$ is explosive, then $Qu=u$ for some bounded non-negative non-zero $u$. Let $w=u|_{\partial}$, i.e., $w_x=u_x$ for all $x\in\partial$. Since $Q|_{\partial}$ is a lower-triangular matrix with non-positive diagonal entries, and $w\ge0$, it is readily deduced that $w=0$ by Gaussian elimination. This implies from $Qu=u$ that $\widetilde{Q}v=v$ with $v=u|_{\cX\setminus\partial}$. Hence $Z_t$ is explosive. To sum up, $Z_t$ is explosive if and only if $Y_t$ is explosive.

\smallskip

Based on the above analysis, it remains to prove the conclusions for case {\rm(a)} using Propositions~\ref{Spro-3} and \ref{Spro-13}.
We first prove the conclusions  assuming $\Omega$ is finite.

{\rm(i)}
We prove explosivity by Proposition~\ref{Spro-3}. Let the lattice interval $A=[0,x_0-\min\cT_-[_1$ for some $x_0>1$ to be determined. Since $\#\cT_-<\infty$, $A\subseteq\N_0$ is finite. Let $f$ be decreasing and bounded such that $f(x)=\mathbbm{1}_{[0,x_0[_1}(x)+x^{-\delta}\mathbbm{1}_{[x_0,\infty[_1}(x)$ for all $x\ge x_0$, with $\delta>0$ to be determined. Obviously, Proposition~\ref{Spro-3}{\rm(i)} is satisfied for the set $A$. Next we verify the conditions in Proposition~\ref{Spro-3}{\rm(ii)}. It is easy to verify by straightforward calculation that
$$Qf(x)<-\epsilon,\ \text{for all}\ x\in \N_0\setminus A,$$
where $\epsilon=\delta\alpha/2$ provided {(C1)} holds with $\delta<R-1$, or  $\epsilon=\delta\lt(\beta-\delta\vartheta\rt)/2$ provided {(C2)}  holds with $\delta<\min\{\beta/\vartheta,R-2\}$, and $x_0$ is chosen large enough. Since $\delta>0$ can be arbitrarily small, in either case, there exist $\delta$ and $\epsilon$ such that the conditions in Proposition~\ref{Spro-3} are fulfilled, and thus $\E_x\zeta<+\infty$ for all $x\in\cX$. In particular, $Y_t$ is explosive a.s..

{\rm(ii)} Now we prove non-explosivity using Proposition~\ref{Spro-13}. Let $f(x)=\log\log(x+1)$ and $g(x)=(|\alpha|+|\beta|+1)(x+M)$ for all $x\in\N_0$ with some $M>0$ to be determined. One can show that all the conditions in Proposition~\ref{Spro-13} are satisfied with some large constant $M>0$, provided neither {(C1)} nor {(C2)} holds.
Hence $Y_t$ is non-explosive.

\subsection{{Proof of Theorem ~\ref{cor-infinite-explosivity}}}

{We first prove explosivity under condition (C1).} Let $f$ be as {in the proof of Theorem~\ref{th-7}(i)}, and let
$$\alpha_-=\lim_{x\to\infty}\tfrac{\sum_{\om\in\cT_-}\lambda_{\om}(x)\om}{x^R}, \quad\alpha_+=\lim_{x\to\infty}\tfrac{\sum_{\om\in\cT_+}\lambda_{\om}(x)\om}{x^R}.$$
 Then $\alpha=\alpha_++\alpha_-$. Since $\alpha>0$, we have $R_+=R$ and there exists $\epsilon_0\in]0,1[$ such that $\alpha_-+(1-\epsilon_0)\alpha_+>0$. By Proposition~\ref{Spro-1}, there exists $N_0,\ u'\in\N$ such that
\[\frac{\sum_{\om\in\cT_+\cap[1,N_0]_1}\lambda_{\om}(x)\om}{\sum_{\om\in\cT_+}\lambda_{\om}(x)\om}\ge1-\epsilon_0,\quad \text{for all}\ x\ge u'.\]
By ($\rm\mathbf{A3}$), $\cT\setminus]N_0,\infty[_1$ is finite. Hence choosing $x_0\ge u'$ large, we have for all $x\in\cX\setminus A$,
\[\begin{split}
Qf(x)&=\sum_{\om\in\cT_-}\lambda_{\om}(x)((x+\om)^{-\delta}-x^{-\delta})+\sum_{\om\in\cT_+}\lambda_{\om}(x)((x+\om)^{-\delta}-x^{-\delta})\\
&\le\sum_{\om\in\cT_-}\lambda_{\om}(x)((x+\om)^{-\delta}-x^{-\delta})+\sum_{\om\in\cT_+\cap[1,N_0]_1}\lambda_{\om}(x)((x+\om)^{-\delta}-x^{-\delta})\\
&\le x^{-\delta}\sum_{\om\in\cT_-}\lambda_{\om}(x)(-\om\delta x^{-1}+\rO(x^{-2}))+x^{-\delta}\sum_{\om\in\cT_+\cap[1,N_0]_1}\lambda_{\om}(x)(-\om\delta x^{-1}+\rO(x^{-2}))\\
&=-\delta\alpha_-x^{R-1-\delta}+\rO(x^{R-2-\delta})-\delta x^{-1-\delta}\sum_{\om\in\cT_+\cap[1,N_0]_1}\lambda_{\om}(x)\om\\
&\le-\delta\alpha_-x^{R-1-\delta}+\rO(x^{R-2-\delta})-\delta (1-\epsilon_0)x^{-1-\delta}\sum_{\om\in\cT_+}\lambda_{\om}(x)\om\\
&=-\delta(\alpha_-+(1-\epsilon_0)\alpha_+)x^{R-1-\delta}+\rO(x^{R-2-\delta})<-\epsilon,
\end{split}\]
where $\epsilon=\frac{\delta(\alpha_-+(1-\epsilon_0)\alpha_+)}{2}$, and $\delta<R-1$. The rest of the argument is the same as {that   of the proof of Theorem~\ref{th-7}(i).}

{Next, we prove non-explosivity under (C3), (C4), or (C5).} Let $f$ and $g$ be as in (ii) in the proof of Theorem~\ref{th-7}. By ($\rm\mathbf{A3}$), for some large $M>0$ to be determined, for all $x\in\N_0$,
\begin{align*}
 {Qf(x)}&=\sum_{\om\in\cT_-}\lambda_{\om}(x)(\log\log(x+1+\om)-\log\log(x+1))\\
&\quad+\sum_{\om\in\cT_+}\lambda_{\om}(x)(\log\log(x+1+\om)-\log\log(x+1))\\
&=\sum_{\om\in\cT_-}\lambda_{\om}(x)\log\Bigr(1+\frac{\log\bigl(1+\frac{\om}{x+1}\bigr)}{\log(x+1)}\Bigr)+
\sum_{\om\in\cT_+}\lambda_{\om}(x)\log\Bigr(1+\frac{\log\bigl(1+\frac{\om}{x+1}\bigr)}{\log(x+1)}\Bigr)\\
&=\sum_{\om\in\cT_-}\lambda_{\om}(x)\Bigl(\frac{\om}{(x+1)\log(x+1)}+\rO\bigl((x+1)^{-2}(\log (x+1))^{-1}\bigr)\Bigr)\\
&\quad+\sum_{\om\in\cT_+}\lambda_{\om}(x)\log\Bigr(1+\frac{\log\bigl(1+\frac{\om}{x+1}\bigr)}{\log(x+1)}\Bigr)\\
& \le\sum_{\om\in\cT_-}\lambda_{\om}(x)\Bigl(\frac{\om}{(x+1)\log(x+1)}+\rO\bigl((x+1)^{-2}(\log (x+1))^{-1}\bigr)\Bigr)\\
&\quad+\sum_{\om\in\cT_+}\lambda_{\om}(x)\frac{\om}{(x+1)\log(x+1)}\\
&=\frac{1}{(x+1)\log(x+1)}\sum_{\om\in\cT}\lambda_{\om}(x)\om+\rO\bigl((x+1)^{R-2}(\log (x+1))^{-1}\bigr)\\
&=\alpha\frac{(x+1)^{R-1}}{\log(x+1)}+\rO\bigl((x+1)^{R-2}(\log (x+1))^{-1}\bigr)\le g(f(x)),
\end{align*}
provided {(C3), (C4), or (C5)} holds.
The rest of the proof is  {the same as that  of   Theorem~\ref{th-7}(ii).}

\subsection{{Proof of Theorem~\ref{th-8}}}

Let $h_A=\bP_{Y_0}(\tau_A<\infty)$ be the \emph{hitting probability} \cite{N98}. In particular, $h_A$ is called the \emph{absorption probability} if $A$ is a closed communicating class. To verify conditions for certain absorption requires the following property for hitting probabilities. For any set $A\subseteq\cX$, we write $h_A(i)$ for $h_A$, to emphasize the dependence of the hitting probability on the initial state $i\in\cX$. In particular, if $A=\{x\}$ is a singleton, we simply write $h_x$ for $h_A$.

Assume without loss of generality that $\cX=\N_0$, and $\partial\subseteq\{0\}$.

\noindent {\rm(i)} We first show recurrence and transience.
The idea {of applying the classical semimartingale approach} originates from \cite{L60}. It suffices to show recurrence and transience for the embedded discrete time Markov chain $\wY_n$ of $Y_t$.

To show recurrence, let $Z_n=\log\log(\widetilde{Y}_n+1)$. Since one-to-one bicontinuous transformation of the state space preserves the Markov property and recurrence, it suffices to show recurrence for $Z_n$.
In the light of the expression for transition probability of $\wY_n$, we have \[\E(Z_{n+1}-Z_n|Z_n=\log\log(x+1))=\frac{1}{\sum_{\om\in\cT}\lambda_{\om}(x)}\sum_{\om\in\cT}\lambda_{\om}(x)(\log\log(x+\om)-\log\log x).\]
By tedious but straightforward computation, we have the following asymptotic expansion:
\[\E(Z_{n+1}-Z_n|Z_n=\log\log(x+1))=\frac{\alpha x^R+\beta x^{R-1}-\vartheta x^{R-1}(\log x)^{-1}+\rO(x^{R-2})}{(1+x)\log(1+x)\sum_{\om\in\cT}\lambda_{\om}(x)}.\] From this asymptotic expansion and note that $\vartheta>0$, we have
\[\E(Z_{n+1}-Z_n|Z_n=\log\log(x+1))\le0,\quad \forall n\in\N_0,\ \text{for all large}\ x,\]
provided {either (C3) or (C6) holds}. From Proposition~\ref{Spro-L1} it follows the recurrence of $Z_n$, and thus recurrence of $\wY_n$ as well.

Next, we prove for transience of $\wY_n$ under reverse conditions {(that is, neither (C3) nor (C6) holds)}.  Let $Z_n'=1-(1+\wY_n)^{-\delta}$ with $\delta>0$ to be determined. Again, $Z_n'$ is a Markov chain, and $\wY_n\to\infty$ if and only if $Z_n'\to1$, which implies \eqref{sup} is fulfilled for $Z_n'$ with $M=1$ since $\wY_n$ on a subset of $\N_0$ is irreducible. Similar as the above computation, we have the asymptotic expansion
\[\E(Z_{n+1}'-Z_n'|Z_n'=1-(1+x)^{-\delta})=\frac{\delta}{(1+x)^{\delta+1}\sum_{\om\in\cT}\lambda_{\om}(x)}(\alpha x^R+(\beta-\delta\vartheta)x^{R-1}+\rO(x^{R-2})).\] Hence
\[\E(Z_{n+1}'-Z_n'|Z_n'=1-(1+x)^{-\delta})\ge0,\quad \forall n\in\N_0,\]
for all large $x$ (and so for all values of $z=Z_n'$ in some interval $C\le z<1$), provided $\alpha>0$ or $\alpha=0,\ \beta>0$ with $\delta<\frac{\beta}{\vartheta}$. By Proposition~\ref{Spro-L2},
$$\mathbb{P}(\lim_{n\to\infty}Z_n'=1)=1,$$
that is,
$$\mathbb{P}(\lim_{n\to\infty}\wY_n=\infty)=1,$$ meaning $\wY_n$ is transient.

\medskip

 {\rm(ii)}  Let $\widetilde{\om}\in\cT_+$. Let $k_0=\min\{l\in\N\colon l\widetilde{\om}\in\N\}$. Define $Z_t$ to be a CTMC on $\N_0$ with transition matrix $\widetilde{Q}=(\widetilde{q}_{xy})$ satisfying for all $x\neq y,\ x,y\in\N_0$,\[\widetilde{q}_{xy}=\cab q_{xy},\quad \text{if}\  x\in\N,\\ 1,\quad\ \ \, \text{if}\ x=0\ \text{and}\ y=j\widetilde{\om},\ j=1,\ldots,k_0,\\ 0,\quad\ \,\ \text{else}.\cae\] It is easy to verify that $Z_t$ is irreducible on $\N_0$. In the following, we show the recurrence of $Z_t$ is equivalent to the absorption of $Y_t$, which yields the conclusion.

On one hand, applying (i) to $Z_t$, we have $Z_t$ is recurrent if and only if {(C3) or (C6)} holds. On the other hand, from Proposition~\ref{pro-2}, $Z_t$ is recurrent if and only if $h^Z_{0}(i)=1$ for all $i\in\cX$, where $h^Z_{0}(i)$ is the hitting probability for $Z_t$. By Proposition~\ref{pro-1}, $h^Z_{0}(i)=1$ for all $i\in\cX$ if and only if $(1,\ldots,1)$ is the minimal non-negative solution to the linear equations
\[\cab x_i=1,\qquad\qquad\qquad\qquad\quad\ i=0,\\ \sum_{j\in\partial\setminus\{i\}}\widetilde{q}_{ij}(x_i-x_j)=0,\quad i\in\N,\cae\] which, by the definition of $\widetilde{Q}$, are identical to
\eqb\label{Eq-2}\cab x_i=1,\qquad\qquad\qquad\qquad\quad\ i=0,\\ \sum_{j\in\cX\setminus\{i\}}q_{ij}(x_i-x_j)=0,\quad i\in\N.\cae\eqe
By Proposition~\ref{pro-1}, $Y_t$ has certain absorption if and only if $(1,\ldots,1)$ is the minimal non-negative solution to \eqref{Eq-2}. Hence the recurrence of $Z_t$ is equivalent to the certain absorption of $Y_t$.

\subsection{{Proof of Theorem~\ref{cor-infinite-recurrence}}} {When $\Omega$ is infinite, then due to $\rm(\mathbf{A3})$, also $\Omega_+$ is infinite while $\Omega_-$ is  finite. Hence, the asymptotic expansions of the sum over all negative jumps in $\Omega_-$ in the proof of Theorem~\ref{th-8} remain valid, while the asymptotic expansions of the sum over all positive jumps in $\Omega_+$ might fail, in that the sum is infinite. Nevertheless, an upper estimate of $Qf(x)$ for certain Lyapunov functions is still possible, as demonstrated in the proof of Theorem~\ref{cor-infinite-explosivity}. A careful examination of the arguments of Theorem~\ref{th-8} shows that the desired upper estimates of $Qf(x)$ hold under the respective conditions listed in Theorem~\ref{cor-infinite-recurrence}, by replacing the asymptotic expansions by one-sided inequalities.}

\subsection{{Proof of Theorem~\ref{th-11}}}

Assume without loss of generality that $\cX=\N_0$.

We first prove the existence of moments of hitting times assuming $\cT$ is finite, by applying Proposition~\ref{Spro-7}(i) case by case.
{We now prove the existence of moments under {(C13) for $0<\delta<1/2$}. Let $f(x)=\sqrt{x+1}$ for $x\in\N_0$.
One can directly verify that for every $0<\sigma=2\delta<2$}, there exists $0<c<+\infty$ such that $Q f^{\sigma}(x)\le -cf^{\sigma-2}(x)$ for all large $x$. By Proposition~\ref{Spro-7}{\rm(i)}, there exists $a>0$ such that
\begin{equation*}
\E_x(\tau^{\epsilon}_{\{f\le a\}})<+\infty,\quad \forall x\in\cX,\ \forall 0<\epsilon<\sigma/2.
\end{equation*}
Moreover, $\{f\le a\}$ is finite since $\lim_{x\to\infty}f(x)=+\infty$.

{Analogous arguments apply to the cases (C16) for $0<\delta<1$ and (C10) for $\delta>0$ with $f(x)=\log (x+1)$; condition (C14) for $0<\delta<\frac{\beta}{\beta-\gamma}$ with $f(x)=\sqrt{x+1}$; condition (C9) with $f(x)=\log\log(x+1)$, and condition  (C3) with $f(x)=x+1$.}

\smallskip

Next we prove the non-existence of hitting times by Proposition~\ref{Spro-7}{\rm(ii)} assuming $\cT$ is finite. For all cases, let $f(x)=g(x)$, and specifically, let $f(x)=x+1$ in the cases {(C13) for $\delta>1$ and (C15) for $\delta>\tfrac{\beta}{\beta-\gamma}$,} and $f(x)=\log(x+1)$ in  case {(C16) for $\delta>1$}. Note that in case {(C15) for $\delta>\tfrac{\beta}{\beta-\gamma}$, we have} $-\frac{\beta}{\vartheta}<1$ is equivalent to $\gamma>0$. The tedious but straightforward verification of the conditions  (ii-1)-(ii-4) in Proposition~\ref{Spro-7} is left to the interested reader.

\subsection{{Proof of Theorem ~\ref{cor-infinite-passagetime}}}

{As alluded to in the proof of Theorem~\ref{cor-infinite-recurrence}, the desired upper estimate of $Qf(x)$ with the same $f$ under (C3) still holds as in the proof of Theorem~\ref{th-11}, by replacing the asymptotic expansions by one-sided inequalities.}

\subsection{{Proof of Theorem~\ref{th-9}}}

Assume without loss of generality that $\cX=\N_0$ and $\partial\subseteq\{0\}$. We prove this theorem by Propositions~\ref{Spro-4}-\ref{Spro-5}, assuming $\cT$ is finite.
 We emphasize that the non-existence of a QSD only rests on the failure of certain absorption.

{\rm(i)} Since $Y_t$ is recurrent, we have $\wY_n$ is recurrent by Proposition~\ref{Spro-N}. Let $\widetilde{\pi}$ be its unique (up to a scalar) invariant measure.
By \cite[Theorem~3.5.1]{N98}, and $\sum_{\om\in\cT}\lambda_{\om}(x)>0$ is bounded away from zero uniformly in $x$,
\begin{equation}\label{Eq-Relation}
\pi(x)=\frac{\widetilde{\pi}(x)}{q_x}=\frac{\widetilde{\pi}(x)}{\sum_{\om\in\cT}\lambda_{\om}(x)},\quad x\in\N_0,\end{equation}
 is a finite stationary measure for $Y_t$. Let $$b=\lim_{x\to\infty}\frac{\sum_{\om\in\Omega}\lambda_\om(x)}{x^R}=\sum_{d_{\om}=R}a_{\om}.$$

Note that when $R=0$, we have $\gamma=0$, and hence $R=0$, $\gamma\neq0$ and $\beta\le0$ cannot occur (see Table~\ref{table-summary}). Due to Theorem~\ref{th-8}, we only need to show

\begin{enumerate}
  \item[$\bullet$] $Y_t$ is positive recurrent if one of {(C3), (C9), (C10), (C11)} holds.
  \item[$\bullet$] $Y_t$ is null recurrent {if one of the following conditions holds
   \begin{enumerate}
     \item[$\circ$] $R=0$, $\alpha=0$;
     \item[$\circ$] $R=1$, $\alpha=\gamma=0$;
     \item[$\circ$] $R=2$, $\alpha=0$, $\beta=0$;
\item[$\circ$] $R=1$, $\alpha=0$, $\beta\le 0\le\gamma$.
   \end{enumerate}}
   \end{enumerate}

First, we show positive recurrence {and exponential ergodicity}. {Notice that (C9)$\cup$(C10)}\\ 
\noindent{$=$(C18)$\cup$(C21). Moreover, (C3) can be decomposed into (C17) and} {\noindent (C17)'  $R=0$, $\alpha<0$.}

For {(C11) or (C17)'}, let $f(x)=x+1$; for {(C21)}, let $f(x)=\log(x+1)$. Then one can verify (using the asymptotics of $Qf$ in Appendix~\ref{subsec-asymptotics})  that there exists $\epsilon>0$ such that $Q f(x)\le-\epsilon$ for all large $x$.
With an appropriate finite set $F$, by Proposition~\ref{Spro-9}, $X_t$ is positive recurrent and there exists a unique ergodic stationary distribution on $\N_0$.

{For {(C17)}, let $f(x)=x+1$; for {(C18)}, let $f(x)=\log\log(x+1)$.} Then one can similarly verify that there exists $\epsilon>0$ such that $Q f(x)\le-\epsilon f(x)$ for all large $x$. By Proposition~\ref{Spro-4}, $Y_t$ is positive recurrent and there exists a unique {\em exponentially} ergodic stationary distribution on $\N_0$.

Next, we show null recurrence case by case, applying Proposition~\ref{prop-non-summability} as well as non-existence of moments of passage times in Theorem~\ref{th-11}.

$\bullet$ {Assume $R=0$, $\alpha=0$.} Then $$\lambda_{\om}(x)\equiv a_{\om},\quad \om\in\cT,\quad \sum_{\om\in\cT}a_{\om}\om=0,$$ since $\lambda_{\om}$ are polynomials.
    Let $c=-\min\cT$, $g(x)=x+c$, $h(x)=(x+c)^{1/2}$, and $f(x)=1$. Hence it is easy to verify that Proposition~\ref{prop-non-summability}(i) and (iii) hold. Let $\mathcal{F}_n$ be the filtration $\widetilde{Y}_n$ is adapted to. Then
 \begin{align*}
   \E(h(\widetilde{Y}_{n+1}-h(\widetilde{Y}_n)|\mathcal{F}_n))
   =&\sum_{\om\in\cT}a_{\om}(\sqrt{\widetilde{Y}_n+\om+c}-\sqrt{\widetilde{Y}_n+c})\\
   =&\sum_{\om\in\cT}\frac{a_{\om}\om}{\sqrt{\widetilde{Y}_n+\om+c}+\sqrt{\widetilde{Y}_n+c}}\\
   =&\sum_{\om\in\cT}\left(\frac{a_{\om}\om}{\sqrt{\widetilde{Y}_n+\om+c}+\sqrt{\widetilde{Y}_n+c}}
   -\frac{a_{\om}\om}{\sqrt{\widetilde{Y}_n+c}+\sqrt{\widetilde{Y}_n+c}}\right)\\
   =&\sum_{\om\in\cT}\frac{\sqrt{\widetilde{Y}_n+c}-\sqrt{\widetilde{Y}_n+\om+c}}{2\sqrt{\widetilde{Y}_n+c}
   (\sqrt{\widetilde{Y}_n+\om+c}+\sqrt{\widetilde{Y}_n+c})}
   a_{\om}\om\\
   =&\frac{1}{2\sqrt{\widetilde{Y}_n+c}}\sum_{\om\in\cT}\frac{a_{\om}\om^2}{(\sqrt{\widetilde{Y}_n+\om+c}
   +\sqrt{\widetilde{Y}_n+c})^2}\ge0,
 \end{align*}
 which shows Proposition~\ref{prop-non-summability}(ii) also holds for finite set $A=[0,c]_1$. Moreover, \[\E(\widetilde{Y}_{n+1}-\widetilde{Y}_n|\mathcal{F}_n)=\sum_{\om\in\cT}a_{\om}\om\equiv0<1=f(x),\quad P_z\text{-a.s.}\quad \text{on}\ \{\tau_A>n\},\] which shows Proposition~\ref{prop-non-summability}(iv) also holds. Since $\pi(j)=\frac{\widetilde{\pi}(j)}{\sum_{\om\in\cT}a_{\om}}$, by Proposition~\ref{prop-non-summability}, we have $\sum_{j\in\N_0}\pi(j)=\infty$.  By the uniqueness of stationary measures under recurrence condition \cite{M63}, we know $Y_t$ is null recurrent.

$\bullet$ {Assume $R=1$, $\alpha=\gamma=0$.} Then $\E(\wY_{n+1}-\wY_n|\wY_n=x)
    =\frac{\sum_{\om\in\cT}\lambda_{\om}(x)\om}{\sum_{\om\in\cT}\lambda_{\om}(x)}\equiv0$ for all large $x$. Applying \cite[Theorem~3.5(iii)]{AI99} with $\varepsilon_0=1/2$ and $a>1$, we have \eqb\label{asy}\sum_{j\ge2}\frac{\widetilde{\pi}(j)}{j\sqrt{\log j}}=\infty.\eqe From \eqref{Eq-Relation}, \eqref{asy} and $\deg(\sum_{\om\in\cT}\lambda_{\om}(x))=1$ for large $x$, it follows that \[\sum_{j\ge2}\frac{\pi(j)}{\sqrt{\log j}}=\infty,\] which implies that $\sum_{j\ge2}\pi(j)=\infty$. The rest argument is the same as in the case $R=0$, $\alpha=0$.

$\bullet$  {Assume $R=2$, $\alpha=\beta=0$.}  Let $g_{\delta}(x)=(\log x)^{\delta}$ for some $1<\delta\le2$. Choose $g=g_{2}$, $h=g_{3/2}$, and $A=[0,a]$ for some large $a>0$ to be determined. Then it is straightforward to verify (using the asymptotics of $Qf$ in Appendix~\ref{subsec-asymptotics}) that $$\E(g_{\delta}(\widetilde{Y}_{n+1})-g_{\delta}(\widetilde{Y}_n)|\mathcal{F}_n)
    =\delta(\delta-1)\frac{\vartheta}{b} x^{-2}(\log x)^{\delta-2}+\rO(x^{-3}(\log x)^{\delta-1}).$$ Hence it is easy to show that $h$ and $g$ satisfies Proposition~\ref{prop-non-summability}(i)-(iii).

Moreover, it is easy to show that there exists $C= C(\vartheta,b)>0$ and $f(x)=Cx^{-2}$ such that
\[\E(g(\widetilde{Y}_{n+1}-g(\widetilde{Y}_n)|\mathcal{F}_n))\le f(\widetilde{Y}_n)\quad \bP_x\text{-a.s.}\quad \text{on}\quad \{\tau_A>n\}.\] By Proposition~\ref{prop-non-summability}, \[\sum_{j\in\cX}f(j)\widetilde{\pi}(j)=\infty.\] Since $\sum_{\om\in\cT}\lambda_{\om}(x)=bx^2+\rO(x)$, substituting \eqref{Eq-Relation} we obtain: $$\sum_{j\in\cX}\pi(j)=\infty.$$ Hence $Y_t$ is null recurrent.
 
$\bullet$ {Assume $R=1$, $\alpha=0$, $\gamma>0$.} To prove null recurrence, it suffices to show that there exists $x\in\N_0$ such that $\E_x(\tau^+_{x})=\infty$. Let $B\subsetneq\N_0$ be as in Theorem~\ref{th-11}{(ii)} and $x=\max B\in\N_0$. Hence $Y_{J_1}\in\cT+x\subsetneq\N_0$ a.s. and $(\cT_++x)\cap B=\varnothing$. By the Markov property of $Y_t$, \[\E_x(\tau_x-J_1|Y_{J_1}=j)=\E_j(\tau_x),\quad \forall j\in\N_0\setminus\{x\}.\] Hence by the law of total probability,
\begin{align*}
\E_x(\tau_x^+)& =\E_x(J_1)+\sum_{j\in\N_0\setminus B}\E_j(\tau_x)\bP_x(Y_{J_1}=j)+\sum_{j\in B}\E_j(\tau_x)\bP_x(Y_{J_1}=j)\\
&\ge\sum_{j\in\N_0\setminus B}\E_j(\tau_x)\bP_x(Y_{J_1}=j)\\
&\ge\bP_x(Y_{J_1}\in\cT_++x)\inf_{j\in\N_0\setminus B}\E_j(\tau_x)\\
&=\frac{\sum_{\om\in\cT_+}\lambda_{\om}(x)}{\sum_{\om\in\cT}\lambda_{\om}(x)}\cdot\infty=\infty,
\end{align*}
since $\frac{\sum_{\om\in\cT_+}\lambda_{\om}(x)}{\sum_{\om\in\cT}\lambda_{\om}(x)}>0$, and $\E_j(\tau_B)=\infty$ for all $j\in\N_0\setminus B$, by Theorem \ref{th-11}{\rm(ii)}, under respective conditions.

{\rm(ii)} By assumption, $\partial=\{0\}$ and $\cX\setminus\partial=\N$. We first show non-existence of QSDs. Construct an irreducible process $Z_t$ on $\N_0$ with transition rate matrix $\widetilde{Q}$ as in the proof of Theorem~\ref{th-8}. Applying conclusion (i) to $Z_t$, $Z_t$ is \emph{not} positive recurrent when none of the conditions of {(C3), (C9), (C10), (C11)} holds. It thus suffices to show that the existence of QSD for $Y_t$ implies positive recurrence of $Z_t$. Assume that $Y_t$ has a QSD on $\N$. By Proposition~\ref{pro-3}, there exists $\psi>0$ such that $$\psi\E_i(\tau_{0})=\E_i(\psi\tau_{0})\le\E_i(\exp(\psi\tau_{0}))<\infty,\quad \forall i\in\N_0.$$ Let $\E_i^Z(\tau_{0})$ be the expected hitting times for process $Z_t$. By Proposition~\ref{pro-4}, $(\E_i^Z(\tau_{0}))_{i\in\N_0}$ is the minimal solution to the associated linear equations with $\widetilde{Q}$. By a similar argument as in the proof of Theorem~\ref{th-8}, $(\E_i^Z(\tau_{0}))_{i\in\N_0}$  is also the minimal solution to the associated linear equations associated with transition matrix $Q$, and thus $$\E_i^Z(\tau_{0})=\E_i(\tau_{0})<\infty,\quad \forall i\in\N_0.$$
Since $Z_t$ is irreducible, we have that $Z_t$ is positive recurrent, due to the classical fact that for an irreducible CTMC that positively recurs to a finite set, positively recurs everywhere (c.f. \cite{MP14}).

Next, we prove the ergodicity {of the $Q$-process}. For {either (C18) or (C19)}, $Y_t$ is non-explosive by Theorem~\ref{th-7}, and let $f(x)=(2-x^{-1})\mathbbm{1}_{\cX\setminus\partial}(x)$.  Under the respective conditions, it is straightforward to verify that $$\lim_{x\to\infty}\frac{Q f(x)}{f(x)}=-\infty,$$ which implies that the set $D=\{x\in \cX\setminus\partial: \frac{Q f(x)}{f(x)}\ge-\psi_0-1\}$ is finite. Then with such $f$, $D$ and $\delta$, the conditions in Proposition~\ref{Spro-5} are satisfied and the conclusions follow. Note that $\supp\nu=\N$ comes from the fact that the support of the ergodic stationary distribution of the $Q$-process is $\N$ by the irreducibility.

\subsection{{Proof of Theorem~\ref{cor-infinite-ergodicity}}} 
{As discussed in the proof of Theorem~\ref{cor-infinite-recurrence}, the desired upper estimates of $Qf(x)$ with the same $f$ under the respective conditions still hold,   by replacing the asymptotic expansions by one-sided inequalities  in the proof of Theorem~\ref{th-9}.}

\subsection{{Proof of Theorem~\ref{th-10}}}

First we prove implosivity. Assume  {(C18) or (C19)} holds. Hence $Y_t$ is recurrent by Theorem~\ref{th-8}. Let $f(x)=1-(x+1)^{-1}$. One can show that the conditions in Proposition~\ref{Spro-6}{\rm(i-1)} are fulfilled, and implosivity is achieved.

Next we turn to non-implosivity. Assume neither {(C18) nor (C19)} holds. Since $Y_t$ does not implode towards any transient state, it suffices to prove non-implosivity assuming recurrence condition, i.e., {(C3) or (C6)}, by Theorem~\ref{th-8}. Let $f(x)=\log\log(x+1)$. It is easy to verify that conditions  (with $\delta=2$) in Proposition~\ref{Spro-6}{\rm(ii)} are fulfilled, and $Y_t$ is non-implosive.

\subsection{{Proof of Theorem~\ref{cor-infinite-implosivity}}} 
{As discussed in the proof of Theorem~\ref{cor-infinite-recurrence}, the same functon $f$ under condition (C19) also serves as a Lyapunov function like in  the proof of Theorem~\ref{th-10}.}

\section*{Acknowledgements}
The authors thank Philip Pollett for making us aware of the reference \cite{R61}, Linard Hoessly for commenting on the manuscript, {as well as the editors' and referees' comments which help improve the presentation of the paper}. The authors acknowledge the support from The Erwin Schr\"{o}dinger Institute (ESI) for the workshop on ``Advances in Chemical Reaction Network Theory''. CX acknowledges the TUM Foundation Fellowship and the Alexander von Humboldt Fellowship. {CW acknowledges support from the Novo Nordisk Foundation (Denmark), grant NNF19OC0058354. }

\newpage

\appendix

\renewcommand\appendixname{Appendix}

\section{Classical criteria for dynamics}\label{appB}

Let $Y_t$ be a CTMC  on state space $\cX\subseteq\N_0$ with transition matrix $Q=(q_{x,y})_{x,y\in\cX}$,
and let $(\wY_n)_{n\in\N_0}$ be its embedded discrete time Markov chain. Let $q_x=\sum_{y\neq x}q_{x,y}$, $\forall x\in\cX$. The transition probability matrix $P=(p_{x,y})_{x,y\in\cX}$ of $\wY_n$ is given by:
\[p_{x,y}=\left\{\begin{array}{cl} q_{x,y}/q_x, & \text{if}\quad x\neq y, q_x\neq0, \\ 0, & \text{if}\quad x\neq y, q_x=0,\end{array}\right.
\qquad p_{x,x}=\left\{\begin{array}{cl} 0, & \text{if}\quad q_x\neq0,\\ 1, & \text{if}\quad q_x=0.\end{array}\right.\]

Let $\mathfrak{F}$ be the set of all non-negative (finite) functions on $\cX$
satisfying
$$\sum_{\om\in\cT}\lambda_{\om}(x)|f(x+\om)|<+\infty,\ \forall x\in\cX.$$
Since $\cX$ is discrete, $\mathfrak{F}$ is indeed a subset of non-negative continuous (and thus Borel measurable) functions on $\cX$. The associated infinitesimal generator is also denoted by $Q$:
\[Qf(x)=\sum_{\omega\in\cT}\lambda_{\om}(x)\lt(f(x+\om)-f(x)\rt),\quad \forall x\in\cX,\quad f\in\mathfrak{F}.\]
By ${\rm(\mathbf{A3})}$,  $\mathfrak{F}$ is a subset of the domain of $Q$. In particular, functions with sub-linear growth rate are in $\mathfrak{F}$. When $\cT$ is finite, $\mathfrak{F}$ is the whole set of all non-negative (finite) functions on $\cX$.

{Before presenting the proofs, we recall   general Lyapunov-Foster type criteria for the reader's convenience \cite{CV17,MP14,MT93}. The proofs  are mainly based on  constructions of specific Lyapunov functions. To avoid tedious but straightforward verifications against the corresponding criteria, we simply provide the specific Lyapunov functions we apply and leave the straightforward verifications   to the interested reader.}

The next proposition is used  to estimate $Qf$ for  a Lyapunov function  $f$.  Let $R_+=\max\{\deg(\lambda_{\om})\colon \om\in\cT_+\}$
 and recall $R=\max\{\deg(\lambda_{\om})\colon \om\in\cT\}$. It hilds that  $R,\ R_+\le M$.

\prob\label{Spro-1}
Assume ${\rm(\mathbf{A1})}$-${\rm(\mathbf{A4})}$.  Let $f_n(x)=\sum_{\om\in\cT_+,\ \om\le n}\lambda_{\om}(x)\om$ for $n\in\N$. Then $f_n$  converges non-decreasingly to a polynomial $f$ of degree $R_+$ on $\cX\setminus[0,u[_1$, \eqb\label{Eq-0}f(x)=\sum_{\om\in\cT_+}\lambda_{\om}(x)\om,\quad x\in\cX\setminus[0,u[_1,\eqe
with $u$  as in ${\rm(\mathbf{A4})}$. Furthermore, $\sum_{\om\in\cT}\lambda_{\om}(x)\om$ is a polynomial of degree at most $R$ on $\cX\setminus[0,u[_1$, and $\sum_{\om\in\cT}\lambda_{\om}(x)$ is a polynomial of degree $R$ on $\cX\setminus[0,u[_1$. Moreover, there exists $u'\ge u$, such that
\eqb\label{Eq-4}\lim_{n\to\infty}\sup_{x\ge u'}\frac{f(x)-f_n(x)}{f(x)}=0.\eqe
\proe

\prb Assume without loss of generality that $u=0$. Otherwise consider $\lambda_{\om}(\cdot+u)$. Furthermore, assume $\cX=\N_0$.
Let $n_*=\min\{\om\in\cT_+\colon \deg(\lambda_{\om})=R_+\}$. Then $(f_n)_{n\ge n_*}$ is a non-decreasing sequence of polynomials on $\N_0$ of degree $R_+$ as the coefficient of $x^{R_+}$ is non-negative in $\lambda_\omega(x)$. By ${\rm(\mathbf{A3})}$-${\rm(\mathbf{A4})}$, $f$ defined in \eqref{Eq-0} is a non-negative finite function on $\N_0$, and $f_n$ converges to $f$ pointwise on $\N_0$.

Write $f_n(x)=\sum_{j=0}^{R_+}\alpha^{(j)}_nx^{\underline{j}}$ as a sum of descending factorials. Since $f_n(j)\to f(j)$ for $j=0,\ldots,R_+$ by assumption, we find inductively in $j$ that $\alpha_n^{(j)}\to\alpha^{(j)}$ for some $\alpha^{(j)}\in\R$,  $j=0,\ldots,R_+$. Let $\widetilde{f}(x)=\sum_{j=0}^{R_+}\alpha^{(j)}x^{\underline{j}}$. Consequently,  $f_n\to \widetilde{f}$ pointwise on $\N_0$, which implies $f=\widetilde f$ and that $f$ is a polynomial on $\N_0$. By  definition of $n_*$ and  monotonicity of  $(f_n)_{n\ge n_*}$, we have $\alpha^{(R_+)}_n\ge \alpha^{(R_+)}_{n_*}>0$ for  $n\ge n_*$, and $\alpha^{(R_+)}=\lim_{n\to\infty}\alpha^{(R_+)}_{n}>0$. Hence $\deg(f)=R_+$. Similarly, by ${\rm(\mathbf{A2})}$, one can show that $\sum_{\om\in\cT}\lambda_{\om}(x)\om$ is a polynomial of degree at most $R$  on $\N_0$, and
$\sum_{\om\in\cT}\lambda_{\om}(x)$ is a polynomial of degree $R$  on $\N_0$. It remains to prove \eqref{Eq-4}.
Indeed, for all $x\in\N$,
\[
  0\le\frac{f(x)-f_n(x)}{f(x)}=\frac{\sum_{j=0}^{R_+}(\alpha^{(j)}-\alpha_n^{(j)})x^{\underline{j}}}{\sum_{j=0}^{R_+}\alpha^{(j)}x^{\underline{j}}}\le \frac{x^{R_+}\sum_{j=0}^{R_+}|\alpha^{(j)}-\alpha_n^{(j)}|}{\sum_{j=0}^{R_+}\alpha^{(j)}x^{\underline{j}}}.
\]
Since there exists $u'\ge u$ such that $f(x)\ge\frac{1}{2}\alpha^{(R_+)}x^{R_+}$ for all $x\ge u'$, then
\[\sup_{x\ge u'}\frac{f(x)-f_n(x)}{f(x)}\le \frac{2\sum_{j=0}^{R_+}|\alpha^{(j)}-\alpha_n^{(j)}|}{\alpha^{(R_+)}},
\] which implies \eqref{Eq-4}.
\pre

\subsection{Criteria for explosivity and non-explosivity}

\prob\label{Spro-3}{\rm\cite[Theorem\,1.12, Remark\,1.13]{MP14}}
Assume $Y_t$ is irreducible on $\cX$. Suppose that there exists a triple $(\epsilon,A,f)$ with a constant $\epsilon>0$, a set $A$ a proper finite subset of $\cX$, such that $\cX\setminus A$ is infinite, and a function $f\in\mathfrak{F}$, such that
\begin{itemize}
\item[\rm(i)] there exists $x_0\in \cX\setminus A$ with $f(x_0)<\min_A f$,
\item[\rm(ii)] $Qf(x)\le-\epsilon$ for all $x\in \cX\setminus A$.
\end{itemize}
Then, the expected life time $\E_x(\zeta)<+\infty$ for all $x\in\cX$.
\proe

\prob\label{Spro-13}{\rm\cite[Theorem\,1.14]{MP14}}
Assume $Y_t$ is irreducible on $\cX$. Let $f\in\mathfrak{F}$ be such that  $\lim_{x\to\infty}f(x)=+\infty$. If
\begin{itemize}
\item[\rm(i)] there exists a non-decreasing function $g\colon[0,\infty[\to[0,\infty[$, such that $G(z)=\int_0^z\frac{{\rm d}y}{g(y)}<+\infty$ for all $z\ge0$ but $\lim_{z\to\infty}G(z)=+\infty$, and
\item[\rm(ii)] $Qf(x)\le g(f(x))$ for all $x\in\cX$,
\end{itemize}then $\mathbb{P}_x(\zeta=+\infty)=1$ for all $x\in\cX$.
\proe

We give Reuter's criterion on explosivity of a CTMC in terms of the transition rate matrix.
\prob\label{Spro-R}{\rm\cite[Theorem\,10]{R57}, {\cite[Theorem\,13.3.11]{B20}}}
Assume $Y_t$ is irreducible on $\cX$ with transition matrix $Q$. Then, $Y_t$ is explosive with positive probability if and only if there exists a nonzero non-negative solution to
\[Qx=\lambda x,\] for some (and all) $\lambda>0$.
\proe

\subsection{Criteria for recurrence, transience and certain absorption}

To prove Theorem~\ref{th-8}(i), we count on the following equivalence regarding recurrence and transience between a CTMC and its embedded discrete time Markov chain.
\prob{\rm\cite[Theorem\,3.4.1]{N98}}\label{Spro-N} Assume that $Y_t$ is irreducible. Let $\widetilde{Y}_n$ be the embedded discrete time Markov chain of $Y_t$. Then
\enb
\item[(i)] $Y_t$ is recurrent if and only if $\wY_n$ is recurrent.
\item[(ii)] $Y_t$ is transient if and only if $\wY_n$ is transient.
\ene
\proe
Apart from the above equivalence, we need the following two properties to prove   recurrence and transience for an irreducible discrete time Markov chain.
\prob{\rm\cite[Theorem\,2.1]{L60}}\label{Spro-L1}
Let $Z_n$ be an irreducible discrete time Markov chain on a subset of $\N_0$. If
\[\E(Z_{n+1}-Z_n|Z_n=x)\le0,\quad \forall n\in\N_0,\ \text{for all large}\ x,\]
then $Z_n$ is recurrent.
\proe
\prob{\rm\cite[Theorem\,2.2]{L60}}\label{Spro-L2}
Let $Z_n$ be a discrete time Markov chain on the real line. Assume that there exists a positive constant $M$ such that\[0\le Z_n<M<\infty,\ \forall n\in\N_0,\] \eqb\label{sup}\mathbb{P}(\limsup_{n\to\infty}Z_n=M)=1.\eqe  If there exists a constant $C<M$ such that
\[\E(Z_{n+1}-Z_n|Z_n=x)\le0,\quad \forall n\in\N_0,\ \text{for all}\ x\ge C,\]
then \[\mathbb{P}(\lim_{n\to\infty}Z_n=M)=1.\]
\proe
 Recall the definition of $\lambda_{\om}$, the transition probabilities {$P=(p_{x,y})_{x,y\in\cX}$} of $\wY_n$ are:
\[p_{x,x+\om}=\frac{\lambda_{\om}(x)}{\sum_{\tilde{\om}\in\cT}\lambda_{\tilde{\om}}(x)}\mathbbm{1}_{\cX\cap(\underset{\tilde{\om}{\in}\cT}{\cup}
\supp\lambda_{\tilde{\om}})}(x),\quad p_{x,x}=1-\mathbbm{1}_{\cX\cap(\underset{\tilde{\om}{\in}\cT}{\cup}
\supp\lambda_{\tilde{\om}})}(x),\ x\in\N_0,\ \om\in\cT.\]

\prob\cite[Theorem~3.3.1]{N98}\label{pro-1}
Let $A\subseteq\cX$. The vector of hitting probabilities $(h_A(i))_{i\in\cX}$ is the minimal non-negative solution to the following linear equations:
\[\cab h_A(i)=1,\qquad\qquad\qquad\qquad\qquad\ \ i\in A,\\ \sum_{j\in\cX\setminus\{i\}}q_{ij}(h_A(i)-h_A(j))=0,\quad i\in\cX\setminus A.\cae\] (Minimality means that if $x$ is another non-negative solution, then $x_i\ge h_A(i)$ for all $i\in\cX_0$.)
\proe
\prob\cite[Theorem~1.5.7,Theorem~3.4.1]{N98}\label{pro-2}
Assume $Y_t$ is irreducible on $\cX$. Then
\enb
\item[(i)] $Y_t$ is recurrent if and only if $h_j(i)=1$ for all $i\in\cX$ and some (and all) $j\in\cX$.
\item[(ii)] $Y_t$ is recurrent if and only if $h_A(i)=1$ for all $i\in\cX$ and some (and all) non-empty subset $A\subseteq\cX$.
\ene
\proe

\prb
Recall that by irreducibility, $Y_t$ is recurrent if and only if one (and every) state $i\in\cX$ is recurrent, which is equivalent to $h_i(i)=1$. Conclusion (i) is a direct result of \cite[Theorem~1.5.7,Theorem~3.4.1]{N98}.

To show (ii), by irreducibility, $\bP_i(\{Y_{\tau_A}=j\})>0$ for all $j\in A$ and $\tau_j=\tau_A$ conditional on $Y_{\tau_A}=j$. Hence, by the law of total probability,
\[\bP_i(\tau_A<\infty)=\sum_{j\in A}\bP_i(\{Y_{\tau_A}=j\})\bP(\tau_j<\infty),\] which implies that $h_A(i)=1$ if and only if $h_j(i)$ for all $j\in A$. On one hand, given any non-empty $A\subseteq\cX$, by (i), since $Y_t$ is recurrent, we have $h_j(i)=1$ for all $i\in\cX$ for all $j\in A$, and thus $h_A(i)=1$. On the other hand, if $h_A(i)=1$ for all $i\in\cX$ and some (and all) subsets $A\subseteq\cX$, then $h_j(i)=1$ for all $j\in A$, and by (i) we know $Y_t$ is recurrent.
\pre

\subsection{Criteria for existence and non-existence of moments of hitting times}

\prob{\rm\cite[Theorem\,1.5]{MP14}}\label{Spro-7}
Assume $Y_t$ is irreducible on $\cX$. Let $f\in\mathfrak{F}$ be such that $\lim_{x\to\infty}f(x)=+\infty$.
\begin{itemize}
\item[\rm(i)]   If there exist positive constants $c_1,\ c_2$ and $\sigma$, such that $f^{\sigma}\in\mathfrak{F}$ and
 $$Q f^{\sigma}(x)\le -c_2f^{\sigma-2}(x),\quad \forall x\in\lt\{f>c_1\rt\},$$
 then $\E_x\lt(\tau^{\epsilon}_{\{f\le c_1\}}\rt)<+\infty$ for all $0<\epsilon<\sigma/2$ and all $x\in\cX$.
\item[\rm(ii)]  Let $g\in\mathfrak{F}$. If there exist
\begin{itemize}
\item[\rm(ii-1)]   a constant $c_1>0$ such that $f\le c_1g$,
\item[\rm(ii-2)]   constants $c_2,\ c_3>0$ such that $Q g(x)\ge -c_3,\quad \forall x\in\lt\{g>c_2\rt\}$,
\item[\rm(ii-3)] constants $c_4>0$ and $\delta>1$ such that $g^{\delta}\in\mathfrak{F}$ and $Q g^{\delta}(x)\le c_4g^{\delta-1}(x),\quad \forall x\in\lt\{g>c_2\rt\}$, and
\item[\rm(ii-4)]  a constant $\sigma>0$ such that $f^{\sigma}\in\mathfrak{F}$ and $Q f^{\sigma}(x)\ge0,\quad \forall x\in\lt\{f>c_1c_2\rt\}$,
\end{itemize}
then $\E_x\lt(\tau^{\epsilon}_{\{f\le c_2\}}\rt)=+\infty$ for all $\epsilon>\sigma$ and all $x\in\lt\{f>c_2\rt\}$.
\end{itemize}
\proe

\subsection{Criteria for positive recurrence, ergodicity, and existence of QSDs}
For the reader's convenience, we first recall the classical Lyapunov-Foster criteria.
\prob{\rm\cite[Theorem\,1.7]{MP14}}\label{Spro-9}
Assume $Y_t$ is irreducible on $\cX$ and recurrent. Then the following are equivalent:
\begin{itemize}
\item[\rm(i)] $Y_t$ is positive recurrent.
\item[\rm(ii)] There exists a triple $(\epsilon,A,f)$, with $\epsilon>0$, $A$ a finite non-empty subset of $\cY$ and $f\in\mathfrak{F}$ verifying $Q f(x)\le -\epsilon$ for all $x\in\cX\setminus A$.
\end{itemize}
\proe
\prob{\rm\cite[Theorem\,7.1]{MT93}}\label{Spro-4}
Assume $Y_t$ is irreducible on $\cX$. Then $Y_t$ is positive recurrent and there exists an exponentially ergodic stationary distribution, if there exists a triple $(\epsilon,A,f)$ with $\epsilon>0$, $A$ a finite subset of $\cX$ and $f\in\mathfrak{F}$ with $\lim_{x\to\infty}f(x)=\infty$, verifying $Q f(x)\le-\epsilon f(x)$ for all $x\notin A$.
\proe
\prob\label{Spro-5}{\rm{\cite[Theorem\,1.1]{CV16},} \cite[Theorem\,5.1, Remark\,11]{CV17}, \cite[Theorem\,2.1]{HZZ19}}
Assume $\partial\neq\varnothing$ and the $Q$-process of $Y_t$ is irreducible. Then, there exists a finite subset $D\subseteq \cX\setminus \partial$, such that $\mathbb{P}_x(Y_1=y) > 0$ for all $x, y\in D$, such  that the constant
$$\psi_0:=\inf\bigl\{\psi\in\R: \liminf_{t\to\infty}e^{\psi t}\mathbb{P}_x(Y_t=x)>0\bigr\}$$
is finite and independent of $x\in D$. If in addition, there exists $\psi_1>\max\{\psi_0,\sup_{x\in \cX\setminus \partial}$ $\sum_{z\in\partial}q_{x,z}\}$, a function $f\in\mathfrak{F}$ such that $f\big|_{\cX\setminus \partial}\ge1$, $f\bigl|_\partial=0$, $\sup\nolimits_{\cX\setminus \partial}f<\infty$, and
\[\sum_{y\in (\cX\setminus \partial)\setminus\{x\}}q_{x,y}f(y)<\infty,\ \forall x\in \cX\setminus \partial;\quad Q f(x)\le-\psi_1f(x),\quad \forall x\in (\cX\setminus \partial)\setminus D,\]
 then there exists a unique {QSD} $\nu$ on $\cX\setminus \partial$ with positive constants $C$ and $\delta<1$, such that for all Borel probability measures $\mu$ on $\cX\setminus \partial$, \[\Bigl\|\mathbb{P}_{\mu}(Y_t\in\cdot|t<\tau_{\partial})-\nu\Bigr\|_{\sf TV}\le C\delta^t,\quad \forall t\ge0.\]
In addition, {${\rm d}\xi(x)=\zeta(x){\rm d}\nu(x)$} is the unique {quasi-ergodic distribution} for $Y_t$, as well as the unique stationary distribution of the $Q$-process, where $\zeta$ is the non-negative function
\[{\zeta}(x)=\lim_{t\to\infty}e^{\psi_0t}\bP_x\lt(t<\tau_{\partial}\rt),\ x\in\cX\setminus \partial.\]
\proe
To show the non-existence of QSDs, we rest on the following two classical results.
\prob\cite[Lemma~4.1]{CMM13}\label{pro-3}
Assume $\partial\neq\varnothing$ and the $Q$-process of $Y_t$ is irreducible. If there exists a QSD for $Y_t$ supported on $\partial^{\sc}$, then the uniform exponential moment property holds:
\[\text{there exists}\ \psi>0\ \text{such that}\ \ \E_x(\exp(\psi\tau_{\partial}))<\infty,\quad \forall x\in\cX.\]
\proe
\prob\cite[Theorem~3.3.3]{N98}\label{pro-4}
Let $A\subseteq\cX$ and $k_A(i)=\E_i(\tau_A)$ for all $i\in\cX$. Assume $q_x\neq0$ for all $x\in\cX\setminus A$. Then the vector of expected hitting times $(k_A(i))_{i\in\cX}$ is the minimal non-negative solution to the following linear equations:
\[\cab k_A(i)=1,\qquad\qquad\qquad\qquad\qquad\ \ {\text{if}}\ \ i\in A,\\ \sum_{j\in\cX\setminus\{i\}}q_{ij}(k_A(i)-k_A(j))=1,\quad {\text{if}}\ i\in\cX\setminus A.\cae\]
\proe

\subsection{Criterion for non-summability of functions with respect to stationary measures}
\begin{proposition}
  \label{prop-non-summability}\cite[Theorem~1', Remarks~3,4]{AI99}
Let $\cX$ be an unbounded countable subset of $\R_{\ge0}$ and $(\mathcal{Y},\mathcal{F},\bP)$ a probability space with a filtration $\{\mathcal{F}_n\}_{n\in\N_0}$. Assume that $Z_n$ is a discrete time $\mathcal{F}_n$-adapted irreducible aperiodic Markov chain on $\cX$, which is recurrent with   unique (up to a multiplicative constant) stationary measure $\nu$. Let $f$ be a non-negative function defined on $\cX$. Then,
 $$\sum_{x\in\cX}f(x)\nu(x)=\infty,$$ 
if there exists some finite set $A$, some $z\in A$, and some non-negative functions $g$ and $h$, such that
\begin{enumerate}
  \item[(i)] $\lim_{x\to\infty}h(x)=\infty$ and $\lim_{x\to\infty}\frac{g(x)}{h(x)}=\infty$,
  \item[(ii)] whenever $z'\in E\subset \cX\setminus A$, the process $\{h(Z_{n\wedge\tau_A})\}_{n\in\N_0}$ is a $\bP_{z'}$-submartingale,
  \item[(iii)] $\E_z(g(Z_n)\mathbbm{1}_{\tau_A>n})$ is finite for all $n\in\N$,
  \item[(iv)] $\E(g(Z_{n+1})-g(Z_n)|\mathcal{F}_n)\le f(Z_n)$, $\bP_z$-a.s., on $\tau_A>n$.
\end{enumerate}
\end{proposition}

\subsection{Criterion for implosivity and non-implosivity}

\prob\cite[Proposition\,2.14]{MP14}\label{Spro-10}
Assume $Y_t$ is irreducible on $\cX$. If there exists a non-empty proper subset $B\subsetneq\cX$ such that $Y_t$ implodes towards $B$, then $Y_t$ is implosive.
\proe

\prob\cite[Theorem\,1.15, Proposition\,1.16]{MP14}\label{Spro-6}
Assume $Y_t$ is irreducible on $\cX$. \begin{itemize}
\item[\rm(i)] The following are equivalent:
\begin{itemize}
\item[\rm(i-1)]  There exists a triple $(\epsilon,F,f)$ with a positive constant $\epsilon$, a finite set $F$, and a function $f\in\mathfrak{F}$ such that $\sup_{x\in\cX}f(x)<+\infty$ and $Q f(x)\le-\epsilon$ whenever $x\in \cX\setminus F$.
\item[\rm(i-2)]  There exists $c>0$, and for every finite $A\subseteq\cX$, there exists a positive {constant $C_A$, such that $\E_x(\tau_A)\le C_A$} and $\E_x(\exp(c\tau_A))<\infty$ whenever $x\in\cX\setminus A$. In particular, $Y_t$ is implosive.
\end{itemize}
\item[\rm(ii)]  Let $f\in\mathfrak{F}$ be such that $\lim_{x\to\infty}f(x)=+\infty$ and assume there exist positive constants $a,\ c,\ \epsilon$ and $\delta>1$ such that $f^{\delta}\in\mathfrak{F}$. In addition, if
$$Q f(x)\ge-\epsilon,\quad Q f^{\delta}(x)\le cf^{\delta-1}(x),\quad \text{whenever}\ x\in \{f>a\},$$
then the chain does not implode towards $\{f\le a\}$.
\end{itemize}
\proe

\subsection{Asymptotic expansion of $Qf$ for Lyapunov   functions $f$ used in the proofs}\label{subsec-asymptotics}

We provide an asymptotic expansion of $Qf(x)$ for all large $x$, for various Lyapunov functions $f$. Let $\delta\in\R$.

\begin{enumerate}
\item[$\bullet$] Let $f(x)=x^{\delta}$. Then \[Qf(x)=\delta x^{\delta}\lt\{\alpha x^{R-1}+(\beta+\delta\vartheta)x^{R-2}+\rO\!\lt(x^{R-3}\rt)\rt\}.\]
\item[$\bullet$] Let $f(x)=(x(\log x)^{-1})^{\delta}$. Then \begin{multline*}
 Qf(x)=\delta (x(\log x)^{-1})^{\delta}\lt\{\alpha\lt(1-(\log x)^{-1}\rt)x^{R-1}\rt.\\
 \lt.+\lt((\beta+\delta\vartheta)-(\beta+2\delta\vartheta)(\log x)^{-1}\rt)x^{R-2}+\rO\!\lt(x^{R-2}(\log x)^{-2}\rt)\rt\}.\end{multline*}
\item[$\bullet$] Let  $f(x)=(x\log x)^{\delta}$. Then \begin{multline*}Qf(x)=\delta (x\log x)^{\delta}\lt\{\alpha\lt(1+(\log x)^{-1}\rt)x^{R-1}+(\beta+\delta\vartheta)x^{R-2}\rt.\\\lt.+(\gamma+\delta\vartheta)x^{R-2}(\log x)^{-1}
+\rO\!\lt(x^{R-2}(\log x)^{-2}\rt)\rt\}.\end{multline*}
\item[$\bullet$] Let  $f(x)=(\log x)^{\delta}$. Then \[Qf(x)=\delta (\log x)^{\delta-1}\lt\{\alpha x^{R-1}+\beta x^{R-2}+(\delta-1)\vartheta x^{R-2}(\log x)^{-1}+\rO\!\lt(x^{R-3}\rt)\rt\}.\]
\item[$\bullet$] Let  $f(x)=(\log\log x)^{\delta}$. Then \begin{multline*}Qf(x)=\delta (\log\log x)^{\delta-1}(\log x)^{-1}\cdot\\ \lt\{\alpha x^{R-1}+\beta x^{R-2}-\vartheta x^{R-2}(\log x)^{-1}+\rO\lt(x^{R-2}(\log x)^{-1}(\log\log x)^{-1}\rt)\rt\}.\end{multline*}
\end{enumerate}

\bibliographystyle{plain}
\footnotesize\bibliography{references}

\end{document}